\documentclass[letterpaper, 12pt]{article}
\usepackage[utf8]{inputenc}
\usepackage{fullpage}
\usepackage{amsmath, amssymb, amsthm, tikz-cd, hyperref}

\newtheorem{thm}{Theorem}[section]
\newtheorem{prop}[thm]{Proposition}
\newtheorem{lem}[thm]{Lemma}
\theoremstyle{definition}
\newtheorem{defn}{Definition}
\newtheorem{obs}[thm]{Observation}
\theoremstyle{remark}
\newtheorem{remark}[thm]{Remark}

\newcommand{\Z}{\mathbb{Z}}
\newcommand{\N}{\mathbb{N}}
\newcommand{\divides}{\ |\ }
\newcommand{\seqnum}[1]{\href{https://oeis.org/#1}{\rm \underline{#1}}}

\usetikzlibrary{arrows}
\tikzset{edge/.style = {->,> = latex'}}
\newcommand{\edge}[2]{\draw[edge,draw=black,thick] (#1) to (#2);}
\tikzset{every picture/.style={line width=0.75pt}}

\def\modd#1 #2{#1\ \mbox{\rm (mod}\ #2\mbox{\rm )}}

\begin{document}

\title{Directed Graphs from Exact Covering Systems }
\author{Dana Neidmann \\
Department of Mathematics \\
University of Illinois at Urbana-Champaign \\ 
Urbana, IL 61801\\
USA \\
\url{dn2@illinois.edu}
\thanks{This work was supported in part by the University of Illinois Urbana-Champaign Research Board Award [RB19143].}}
\date{}

\maketitle

\begin{abstract}
    Given an exact covering system $S = \{a_i$ (mod $d_i$) $: 1 \leq i \leq r\}$, we introduce the corresponding exact covering system digraph (ECSD) $G_S = G(d_1n+a_1, \ldots, d_rn+a_r)$. The vertices of $G_S$ are the integers and the edges are $(n, d_in+a_i)$ for each  $n \in \Z$ and for each congruence in the covering system. We study the structure of these directed graphs, which have finitely many components, one cycle per component, as well as indegree 1 and outdegree $r$ at each vertex. We also explore the link between ECSDs that have a single component and non-standard digital representations of integers. 
\end{abstract}

\section{Introduction}

In this paper, we take inspiration from the fundamental concept of covering systems to create an associated family of directed graphs on the integers.

\begin{defn} For integers $a_i$, $d_i$, and $r$, a system of congruences 
\[\{x \equiv \modd{a_i} {d_i}:  1 \leq i \leq r\}\] 
is called a \textit{covering system} if every integer $n$ satisfies $n \equiv a_i$ (mod $d_i$) for at least one value of  $i$, equivalently we say $n$ is covered by the congruence $x \equiv a_i$ (mod $d_i$). A covering system in which each integer is covered by exactly one congruence is called an \textit{exact covering system}.
\end{defn}

As congruences denote equivalence classes, choosing an equivalent representative $a_i$ does not change the congruence, for example, $x \equiv \modd{1} {2}$ and $x\equiv \modd{3} {2}$ denote the same equivalence class. However, for our purposes in this paper, we do care about the choice of representative. 

Similarly, although number theorists customarily think about congruences modulo positive numbers only, it will make sense for us to consider negative moduli as well: We say that an integer $n$ is congruent to $a$ (mod $-d$) if and only if $n \equiv a$ (mod $d$). This is because $d \divides n-a$ if and only if $-d \divides n-a$. Thus we can think of $-d$ as another representative of the modulus $d$, so the congruences $x \equiv 1$ (mod 2) and $x \equiv 1$ (mod $-2$) denote the same equivalence class. However, for our purposes we will also care about the sign of the modulus. 

To clarify these choices, for a given exact covering system we can introduce a \textit{representative set} denoting the specific choice of representatives. 

\begin{defn}
The set of pairs $S = \{(a_i, d_i) \in \Z^2 : 1 \leq i \leq r\}$ is a \textit{representative set for an exact covering system} if  the system of congruences $\{x \equiv \modd{a_i} {d_i} : 1 \leq i \leq r\}$ is an exact covering system. 
\end{defn}

For each representative set, we introduce a corresponding (infinite) directed graph on the set of integers (note that the definition can be extended to any covering system). 

\begin{defn} Let $S = \{(a_i, d_i) \in \Z^2 : 1 \leq i \leq r\}$ be a representative set for an exact covering system. We denote the corresponding \textit{exact covering system digraph (ECSD)} 
\[G_S = G(d_1n+a_1, \ldots, d_rn + a_r) := (V,E),\] 
where $V(G_S) = \Z$, and $E(G_S) = \{(n, d_in+a_i) : 1 \leq i \leq r\}$. Note that the parameters of $G_S$ are the $r$ integers to which each $n \in \Z$ is sent. Let $r$ be called the \textit{degree} of $G_S$ (this is the outdegree of every vertex). 
\end{defn}

\begin{figure}[h]
    \centering
\begin{tikzpicture}[x=0.75pt,y=0.75pt,yscale=-.75,xscale=.75]

\draw [color={rgb, 255:red, 0; green, 0; blue, 0 }  ,draw opacity=1 ]   (52.5,87) .. controls (31.92,58.58) and (83.37,58.97) .. (70.38,84.41) ;
\draw [shift={(69.5,86)}, rotate = 300.65] [color={rgb, 255:red, 0; green, 0; blue, 0 }  ,draw opacity=1 ][line width=0.75]    (10.93,-3.29) .. controls (6.95,-1.4) and (3.31,-0.3) .. (0,0) .. controls (3.31,0.3) and (6.95,1.4) .. (10.93,3.29)   ;

\draw [color={rgb, 255:red, 0; green, 0; blue, 0 }  ,draw opacity=1 ]   (588.5,83) .. controls (567.92,54.58) and (619.37,54.97) .. (606.38,80.41) ;
\draw [shift={(605.5,82)}, rotate = 300.65] [color={rgb, 255:red, 0; green, 0; blue, 0 }  ,draw opacity=1 ][line width=0.75]    (10.93,-3.29) .. controls (6.95,-1.4) and (3.31,-0.3) .. (0,0) .. controls (3.31,0.3) and (6.95,1.4) .. (10.93,3.29)   ;

\draw (61,97) node    {$0$};
\draw (598,95) node    {$9$};
\draw (60,147) node    {$-9$};
\draw (598,143) node    {$18$};
\draw (45,196) node    {$\vdots $};
\draw (82,196) node    {$\vdots $};
\draw (582,194) node    {$\vdots $};
\draw (618,195) node    {$\vdots $};
\draw (225,148) node    {$1$};

\draw (251,193) node    {$2$};

\draw (308,193) node    {$4$};

\draw (332,145) node    {$8$};

\draw (307,102) node    {$7$};

\draw (248,102) node    {$5$};

\draw (498,122) node    {$3$};

\draw (498,167) node    {$6$};

\draw (179,148) node    {$-7$};

\draw (230,234) node    {$-5$};

\draw (333,234) node    {$-1$};

\draw (378,146) node    {$16$};

\draw (329,63) node    {$14$};

\draw (218,61) node    {$10$};

\draw (496,75) node    {$-3$};

\draw (496,219) node    {$12$};

\draw (137,104) node  [rotate=-318.08]  {$\vdots $};

\draw (137,173) node  [rotate=-249.25]  {$\vdots $};

\draw (181,262) node  [rotate=-236]  {$\vdots $};

\draw (234,281) node  [rotate=-175.05]  {$\vdots $};

\draw (320,280) node  [rotate=-182.84]  {$\vdots $};

\draw (377,248) node  [rotate=-102.14]  {$\vdots $};

\draw (428,180) node  [rotate=-117.9]  {$\vdots $};

\draw (421,113) node  [rotate=-57.97]  {$\vdots $};

\draw (379,57) node  [rotate=-83.95]  {$\vdots $};

\draw (325,19) node  [rotate=-359.26]  {$\vdots $};

\draw (221,18) node  [rotate=-5.3]  {$\vdots $};

\draw (173,51) node  [rotate=-285.94]  {$\vdots $};

\draw (465,31) node  [rotate=-323.16]  {$\vdots $};

\draw (533,34) node  [rotate=-42.76]  {$\vdots $};

\draw (462,262) node  [rotate=-221.46]  {$\vdots $};

\draw (531,262) node  [rotate=-135.33]  {$\vdots $};

\draw [color={rgb, 255:red, 0; green, 0; blue, 0 }  ,draw opacity=1 ]   (598,107) -- (598,129) ;
\draw [shift={(598,131)}, rotate = 270] [color={rgb, 255:red, 0; green, 0; blue, 0 }  ,draw opacity=1 ][line width=0.75]    (10.93,-3.29) .. controls (6.95,-1.4) and (3.31,-0.3) .. (0,0) .. controls (3.31,0.3) and (6.95,1.4) .. (10.93,3.29)   ;

\draw [color={rgb, 255:red, 0; green, 0; blue, 0 }  ,draw opacity=1 ]   (60.76,109) -- (60.28,133) ;
\draw [shift={(60.24,135)}, rotate = 271.15] [color={rgb, 255:red, 0; green, 0; blue, 0 }  ,draw opacity=1 ][line width=0.75]    (10.93,-3.29) .. controls (6.95,-1.4) and (3.31,-0.3) .. (0,0) .. controls (3.31,0.3) and (6.95,1.4) .. (10.93,3.29)   ;

\draw [color={rgb, 255:red, 0; green, 0; blue, 0 }  ,draw opacity=1 ]   (594.24,155) -- (586.36,180.09) ;
\draw [shift={(585.76,182)}, rotate = 287.42] [color={rgb, 255:red, 0; green, 0; blue, 0 }  ,draw opacity=1 ][line width=0.75]    (10.93,-3.29) .. controls (6.95,-1.4) and (3.31,-0.3) .. (0,0) .. controls (3.31,0.3) and (6.95,1.4) .. (10.93,3.29)   ;

\draw [color={rgb, 255:red, 0; green, 0; blue, 0 }  ,draw opacity=1 ]   (602.62,155) -- (612.67,181.13) ;
\draw [shift={(613.38,183)}, rotate = 248.95999999999998] [color={rgb, 255:red, 0; green, 0; blue, 0 }  ,draw opacity=1 ][line width=0.75]    (10.93,-3.29) .. controls (6.95,-1.4) and (3.31,-0.3) .. (0,0) .. controls (3.31,0.3) and (6.95,1.4) .. (10.93,3.29)   ;

\draw [color={rgb, 255:red, 0; green, 0; blue, 0 }  ,draw opacity=1 ]   (56.33,159) -- (49.26,182.09) ;
\draw [shift={(48.67,184)}, rotate = 287.02] [color={rgb, 255:red, 0; green, 0; blue, 0 }  ,draw opacity=1 ][line width=0.75]    (10.93,-3.29) .. controls (6.95,-1.4) and (3.31,-0.3) .. (0,0) .. controls (3.31,0.3) and (6.95,1.4) .. (10.93,3.29)   ;

\draw [color={rgb, 255:red, 0; green, 0; blue, 0 }  ,draw opacity=1 ]   (65.39,159) -- (75.79,182.18) ;
\draw [shift={(76.61,184)}, rotate = 245.82] [color={rgb, 255:red, 0; green, 0; blue, 0 }  ,draw opacity=1 ][line width=0.75]    (10.93,-3.29) .. controls (6.95,-1.4) and (3.31,-0.3) .. (0,0) .. controls (3.31,0.3) and (6.95,1.4) .. (10.93,3.29)   ;

\draw [color={rgb, 255:red, 0; green, 0; blue, 0 }  ,draw opacity=1 ]   (231.93,160) -- (243.07,179.27) ;
\draw [shift={(244.07,181)}, rotate = 239.98] [color={rgb, 255:red, 0; green, 0; blue, 0 }  ,draw opacity=1 ][line width=0.75]    (10.93,-3.29) .. controls (6.95,-1.4) and (3.31,-0.3) .. (0,0) .. controls (3.31,0.3) and (6.95,1.4) .. (10.93,3.29)   ;

\draw [color={rgb, 255:red, 0; green, 0; blue, 0 }  ,draw opacity=1 ]   (260.5,193) -- (296.5,193) ;
\draw [shift={(298.5,193)}, rotate = 180] [color={rgb, 255:red, 0; green, 0; blue, 0 }  ,draw opacity=1 ][line width=0.75]    (10.93,-3.29) .. controls (6.95,-1.4) and (3.31,-0.3) .. (0,0) .. controls (3.31,0.3) and (6.95,1.4) .. (10.93,3.29)   ;

\draw [color={rgb, 255:red, 0; green, 0; blue, 0 }  ,draw opacity=1 ]   (314,181) -- (325.11,158.79) ;
\draw [shift={(326,157)}, rotate = 476.57] [color={rgb, 255:red, 0; green, 0; blue, 0 }  ,draw opacity=1 ][line width=0.75]    (10.93,-3.29) .. controls (6.95,-1.4) and (3.31,-0.3) .. (0,0) .. controls (3.31,0.3) and (6.95,1.4) .. (10.93,3.29)   ;

\draw [color={rgb, 255:red, 0; green, 0; blue, 0 }  ,draw opacity=1 ]   (242,114) -- (231.89,134.21) ;
\draw [shift={(231,136)}, rotate = 296.57] [color={rgb, 255:red, 0; green, 0; blue, 0 }  ,draw opacity=1 ][line width=0.75]    (10.93,-3.29) .. controls (6.95,-1.4) and (3.31,-0.3) .. (0,0) .. controls (3.31,0.3) and (6.95,1.4) .. (10.93,3.29)   ;

\draw [color={rgb, 255:red, 0; green, 0; blue, 0 }  ,draw opacity=1 ]   (297.5,102) -- (259.5,102) ;
\draw [shift={(257.5,102)}, rotate = 360] [color={rgb, 255:red, 0; green, 0; blue, 0 }  ,draw opacity=1 ][line width=0.75]    (10.93,-3.29) .. controls (6.95,-1.4) and (3.31,-0.3) .. (0,0) .. controls (3.31,0.3) and (6.95,1.4) .. (10.93,3.29)   ;

\draw [color={rgb, 255:red, 0; green, 0; blue, 0 }  ,draw opacity=1 ]   (325.02,133) -- (314.98,115.73) ;
\draw [shift={(313.98,114)}, rotate = 419.83000000000004] [color={rgb, 255:red, 0; green, 0; blue, 0 }  ,draw opacity=1 ][line width=0.75]    (10.93,-3.29) .. controls (6.95,-1.4) and (3.31,-0.3) .. (0,0) .. controls (3.31,0.3) and (6.95,1.4) .. (10.93,3.29)   ;

\draw [color={rgb, 255:red, 0; green, 0; blue, 0 }  ,draw opacity=1 ]   (507.5,130.7) .. controls (513.45,138.62) and (513.78,147.04) .. (508.48,155.95) ;
\draw [shift={(507.5,157.51)}, rotate = 303.59000000000003] [color={rgb, 255:red, 0; green, 0; blue, 0 }  ,draw opacity=1 ][line width=0.75]    (10.93,-3.29) .. controls (6.95,-1.4) and (3.31,-0.3) .. (0,0) .. controls (3.31,0.3) and (6.95,1.4) .. (10.93,3.29)   ;

\draw [color={rgb, 255:red, 0; green, 0; blue, 0 }  ,draw opacity=1 ]   (488.5,157.84) .. controls (482.53,149.23) and (482.17,140.78) .. (487.43,132.49) ;
\draw [shift={(488.5,130.9)}, rotate = 485.82] [color={rgb, 255:red, 0; green, 0; blue, 0 }  ,draw opacity=1 ][line width=0.75]    (10.93,-3.29) .. controls (6.95,-1.4) and (3.31,-0.3) .. (0,0) .. controls (3.31,0.3) and (6.95,1.4) .. (10.93,3.29)   ;

\draw [color={rgb, 255:red, 0; green, 0; blue, 0 }  ,draw opacity=1 ]   (315.32,205) -- (324.64,220.29) ;
\draw [shift={(325.68,222)}, rotate = 238.63] [color={rgb, 255:red, 0; green, 0; blue, 0 }  ,draw opacity=1 ][line width=0.75]    (10.93,-3.29) .. controls (6.95,-1.4) and (3.31,-0.3) .. (0,0) .. controls (3.31,0.3) and (6.95,1.4) .. (10.93,3.29)   ;

\draw [color={rgb, 255:red, 0; green, 0; blue, 0 }  ,draw opacity=1 ]   (341.5,145.21) -- (363.5,145.68) ;
\draw [shift={(365.5,145.73)}, rotate = 181.25] [color={rgb, 255:red, 0; green, 0; blue, 0 }  ,draw opacity=1 ][line width=0.75]    (10.93,-3.29) .. controls (6.95,-1.4) and (3.31,-0.3) .. (0,0) .. controls (3.31,0.3) and (6.95,1.4) .. (10.93,3.29)   ;

\draw [color={rgb, 255:red, 0; green, 0; blue, 0 }  ,draw opacity=1 ]   (313.77,90) -- (321.25,76.74) ;
\draw [shift={(322.23,75)}, rotate = 479.43] [color={rgb, 255:red, 0; green, 0; blue, 0 }  ,draw opacity=1 ][line width=0.75]    (10.93,-3.29) .. controls (6.95,-1.4) and (3.31,-0.3) .. (0,0) .. controls (3.31,0.3) and (6.95,1.4) .. (10.93,3.29)   ;

\draw [color={rgb, 255:red, 0; green, 0; blue, 0 }  ,draw opacity=1 ]   (239.22,90) -- (227.96,74.61) ;
\draw [shift={(226.78,73)}, rotate = 413.81] [color={rgb, 255:red, 0; green, 0; blue, 0 }  ,draw opacity=1 ][line width=0.75]    (10.93,-3.29) .. controls (6.95,-1.4) and (3.31,-0.3) .. (0,0) .. controls (3.31,0.3) and (6.95,1.4) .. (10.93,3.29)   ;

\draw [color={rgb, 255:red, 0; green, 0; blue, 0 }  ,draw opacity=1 ]   (244.85,205) -- (237.06,220.22) ;
\draw [shift={(236.15,222)}, rotate = 297.12] [color={rgb, 255:red, 0; green, 0; blue, 0 }  ,draw opacity=1 ][line width=0.75]    (10.93,-3.29) .. controls (6.95,-1.4) and (3.31,-0.3) .. (0,0) .. controls (3.31,0.3) and (6.95,1.4) .. (10.93,3.29)   ;

\draw [color={rgb, 255:red, 0; green, 0; blue, 0 }  ,draw opacity=1 ]   (215.5,148) -- (194.5,148) ;
\draw [shift={(192.5,148)}, rotate = 360] [color={rgb, 255:red, 0; green, 0; blue, 0 }  ,draw opacity=1 ][line width=0.75]    (10.93,-3.29) .. controls (6.95,-1.4) and (3.31,-0.3) .. (0,0) .. controls (3.31,0.3) and (6.95,1.4) .. (10.93,3.29)   ;

\draw [color={rgb, 255:red, 0; green, 0; blue, 0 }  ,draw opacity=1 ]   (497.49,110) -- (496.6,89) ;
\draw [shift={(496.51,87)}, rotate = 447.56] [color={rgb, 255:red, 0; green, 0; blue, 0 }  ,draw opacity=1 ][line width=0.75]    (10.93,-3.29) .. controls (6.95,-1.4) and (3.31,-0.3) .. (0,0) .. controls (3.31,0.3) and (6.95,1.4) .. (10.93,3.29)   ;

\draw [color={rgb, 255:red, 0; green, 0; blue, 0 }  ,draw opacity=1 ]   (497.54,179) -- (496.54,205) ;
\draw [shift={(496.46,207)}, rotate = 272.2] [color={rgb, 255:red, 0; green, 0; blue, 0 }  ,draw opacity=1 ][line width=0.75]    (10.93,-3.29) .. controls (6.95,-1.4) and (3.31,-0.3) .. (0,0) .. controls (3.31,0.3) and (6.95,1.4) .. (10.93,3.29)   ;

\draw [color={rgb, 255:red, 0; green, 0; blue, 0 }  ,draw opacity=1 ]   (205.5,58.22) -- (189.18,54.59) ;
\draw [shift={(187.22,54.16)}, rotate = 372.53] [color={rgb, 255:red, 0; green, 0; blue, 0 }  ,draw opacity=1 ][line width=0.75]    (10.93,-3.29) .. controls (6.95,-1.4) and (3.31,-0.3) .. (0,0) .. controls (3.31,0.3) and (6.95,1.4) .. (10.93,3.29)   ;

\draw [color={rgb, 255:red, 0; green, 0; blue, 0 }  ,draw opacity=1 ]   (218.84,49) -- (219.97,32.83) ;
\draw [shift={(220.1,30.83)}, rotate = 453.99] [color={rgb, 255:red, 0; green, 0; blue, 0 }  ,draw opacity=1 ][line width=0.75]    (10.93,-3.29) .. controls (6.95,-1.4) and (3.31,-0.3) .. (0,0) .. controls (3.31,0.3) and (6.95,1.4) .. (10.93,3.29)   ;

\draw [color={rgb, 255:red, 0; green, 0; blue, 0 }  ,draw opacity=1 ]   (327.91,51) -- (326.28,33.11) ;
\draw [shift={(326.1,31.12)}, rotate = 444.81] [color={rgb, 255:red, 0; green, 0; blue, 0 }  ,draw opacity=1 ][line width=0.75]    (10.93,-3.29) .. controls (6.95,-1.4) and (3.31,-0.3) .. (0,0) .. controls (3.31,0.3) and (6.95,1.4) .. (10.93,3.29)   ;

\draw [color={rgb, 255:red, 0; green, 0; blue, 0 }  ,draw opacity=1 ]   (341.5,61.5) -- (364.07,58.79) ;
\draw [shift={(366.05,58.55)}, rotate = 533.1600000000001] [color={rgb, 255:red, 0; green, 0; blue, 0 }  ,draw opacity=1 ][line width=0.75]    (10.93,-3.29) .. controls (6.95,-1.4) and (3.31,-0.3) .. (0,0) .. controls (3.31,0.3) and (6.95,1.4) .. (10.93,3.29)   ;

\draw [color={rgb, 255:red, 0; green, 0; blue, 0 }  ,draw opacity=1 ]   (390.5,136.41) -- (403.9,126.13) ;
\draw [shift={(405.48,124.91)}, rotate = 502.5] [color={rgb, 255:red, 0; green, 0; blue, 0 }  ,draw opacity=1 ][line width=0.75]    (10.93,-3.29) .. controls (6.95,-1.4) and (3.31,-0.3) .. (0,0) .. controls (3.31,0.3) and (6.95,1.4) .. (10.93,3.29)   ;

\draw [color={rgb, 255:red, 0; green, 0; blue, 0 }  ,draw opacity=1 ]   (390.5,154.5) -- (411.06,168.48) ;
\draw [shift={(412.72,169.61)}, rotate = 214.22] [color={rgb, 255:red, 0; green, 0; blue, 0 }  ,draw opacity=1 ][line width=0.75]    (10.93,-3.29) .. controls (6.95,-1.4) and (3.31,-0.3) .. (0,0) .. controls (3.31,0.3) and (6.95,1.4) .. (10.93,3.29)   ;

\draw [color={rgb, 255:red, 0; green, 0; blue, 0 }  ,draw opacity=1 ]   (346.5,238.3) -- (361.32,243.01) ;
\draw [shift={(363.23,243.62)}, rotate = 197.65] [color={rgb, 255:red, 0; green, 0; blue, 0 }  ,draw opacity=1 ][line width=0.75]    (10.93,-3.29) .. controls (6.95,-1.4) and (3.31,-0.3) .. (0,0) .. controls (3.31,0.3) and (6.95,1.4) .. (10.93,3.29)   ;

\draw [color={rgb, 255:red, 0; green, 0; blue, 0 }  ,draw opacity=1 ]   (329.61,246) -- (324.06,265.62) ;
\draw [shift={(323.52,267.54)}, rotate = 285.78] [color={rgb, 255:red, 0; green, 0; blue, 0 }  ,draw opacity=1 ][line width=0.75]    (10.93,-3.29) .. controls (6.95,-1.4) and (3.31,-0.3) .. (0,0) .. controls (3.31,0.3) and (6.95,1.4) .. (10.93,3.29)   ;

\draw [color={rgb, 255:red, 0; green, 0; blue, 0 }  ,draw opacity=1 ]   (231.02,246) -- (232.74,266.22) ;
\draw [shift={(232.91,268.22)}, rotate = 265.14] [color={rgb, 255:red, 0; green, 0; blue, 0 }  ,draw opacity=1 ][line width=0.75]    (10.93,-3.29) .. controls (6.95,-1.4) and (3.31,-0.3) .. (0,0) .. controls (3.31,0.3) and (6.95,1.4) .. (10.93,3.29)   ;

\draw [color={rgb, 255:red, 0; green, 0; blue, 0 }  ,draw opacity=1 ]   (216.5,241.71) -- (198.34,252.09) ;
\draw [shift={(196.6,253.08)}, rotate = 330.26] [color={rgb, 255:red, 0; green, 0; blue, 0 }  ,draw opacity=1 ][line width=0.75]    (10.93,-3.29) .. controls (6.95,-1.4) and (3.31,-0.3) .. (0,0) .. controls (3.31,0.3) and (6.95,1.4) .. (10.93,3.29)   ;

\draw [color={rgb, 255:red, 0; green, 0; blue, 0 }  ,draw opacity=1 ]   (165.5,156.04) -- (153.44,163.22) ;
\draw [shift={(151.72,164.24)}, rotate = 329.24] [color={rgb, 255:red, 0; green, 0; blue, 0 }  ,draw opacity=1 ][line width=0.75]    (10.93,-3.29) .. controls (6.95,-1.4) and (3.31,-0.3) .. (0,0) .. controls (3.31,0.3) and (6.95,1.4) .. (10.93,3.29)   ;

\draw [color={rgb, 255:red, 0; green, 0; blue, 0 }  ,draw opacity=1 ]   (167.55,136) -- (153.45,121.23) ;
\draw [shift={(152.07,119.79)}, rotate = 406.33000000000004] [color={rgb, 255:red, 0; green, 0; blue, 0 }  ,draw opacity=1 ][line width=0.75]    (10.93,-3.29) .. controls (6.95,-1.4) and (3.31,-0.3) .. (0,0) .. controls (3.31,0.3) and (6.95,1.4) .. (10.93,3.29)   ;

\draw [color={rgb, 255:red, 0; green, 0; blue, 0 }  ,draw opacity=1 ]   (486.51,231) -- (475.72,244.65) ;
\draw [shift={(474.48,246.22)}, rotate = 308.33000000000004] [color={rgb, 255:red, 0; green, 0; blue, 0 }  ,draw opacity=1 ][line width=0.75]    (10.93,-3.29) .. controls (6.95,-1.4) and (3.31,-0.3) .. (0,0) .. controls (3.31,0.3) and (6.95,1.4) .. (10.93,3.29)   ;

\draw [color={rgb, 255:red, 0; green, 0; blue, 0 }  ,draw opacity=1 ]   (505.77,231) -- (516.88,244.66) ;
\draw [shift={(518.15,246.21)}, rotate = 230.86] [color={rgb, 255:red, 0; green, 0; blue, 0 }  ,draw opacity=1 ][line width=0.75]    (10.93,-3.29) .. controls (6.95,-1.4) and (3.31,-0.3) .. (0,0) .. controls (3.31,0.3) and (6.95,1.4) .. (10.93,3.29)   ;

\draw [color={rgb, 255:red, 0; green, 0; blue, 0 }  ,draw opacity=1 ]   (487.55,63) -- (477.21,48.33) ;
\draw [shift={(476.06,46.7)}, rotate = 414.83000000000004] [color={rgb, 255:red, 0; green, 0; blue, 0 }  ,draw opacity=1 ][line width=0.75]    (10.93,-3.29) .. controls (6.95,-1.4) and (3.31,-0.3) .. (0,0) .. controls (3.31,0.3) and (6.95,1.4) .. (10.93,3.29)   ;

\draw [color={rgb, 255:red, 0; green, 0; blue, 0 }  ,draw opacity=1 ]   (506.83,63) -- (517.41,51.28) ;
\draw [shift={(518.75,49.79)}, rotate = 492.06] [color={rgb, 255:red, 0; green, 0; blue, 0 }  ,draw opacity=1 ][line width=0.75]    (10.93,-3.29) .. controls (6.95,-1.4) and (3.31,-0.3) .. (0,0) .. controls (3.31,0.3) and (6.95,1.4) .. (10.93,3.29)   ;

\end{tikzpicture}
    \caption{The ECSD $G(2n, 2n-9)$.}
    \label{1(2n,2n-9)}
\end{figure}
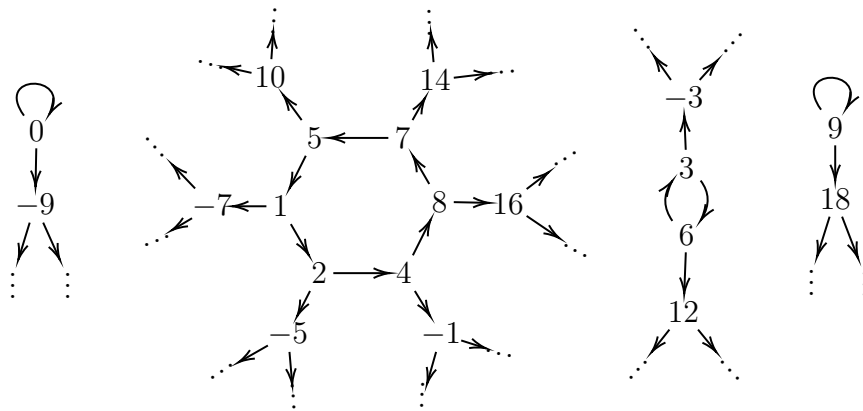

In Figure \ref{1(2n,2n-9)}, we see the ECSD corresponding to the representative set $S = \{(0,2), (-9,2)\}$. We refer to this ECSD as $G(2n, 2n-9)$ because every integer $n$ has a directed edge to $2n$ as well as to $2n-9$. All components are shown, and dotted lines indicate that the graph continues on infinitely. (Note that in this paper, we use \textit{component} to refer to a connected component in the underlying undirected graph, more precisely, a weakly connected component of the digraph.) In this figure, we can see that each $n \in \Z$ has indegree 1; in fact, this is true for each vertex in any ECSD since each integer satisfies exactly one congruence:

\begin{obs} \label{ECSDpredecessorcondition}
Each vertex of an ECSD has indegree 1. 
\end{obs}

\begin{proof}
Given an ECSD $G_S$, each $n \in \Z$ is in exactly one congruence class of the exact covering system corresponding to $S$, say, for example, we have $n \equiv a_i$ (mod $d_i$). Therefore, there is a unique $m$ such that $d_im + a_i = n$, and a single edge $(m,n)$ ending in $n$.
\end{proof}

Since every vertex has indegree 1, we can define several more terms:

\begin{defn}
In an ECSD, let the \textit{predecessor} of $n$ be the unique $P(n) \in \Z$ such that $(P(n), n)$ is an edge. Let $P^k(n) = P(P(\cdots P(n)\cdots))$, where the predecessor function is composed $k$ times.  If $m = P^k(n)$ for some $k \in \N$, we call $m$ a \textit{($k$-)ancestor} of $n$, and call $n$ a \textit{($k$-)descendant} of $m$. We call a 1-descendant of $n$ a \textit{successor} of $n$.
\end{defn}

For example, in Figure \ref{1(2n,2n-9)} we can see that the predecessor of $-5$ is 2, two visible $2$-descendants of 8 are 5 and 14, and 7 is an ancestor of $-1$. \\

\noindent
ECSDs have several interesting applications within number theory. The example that inspired the definition of ECSDs comes from the Stern sequence $s(n)$, defined recursively by 
\[s(0)=0, \; s(1)=1, \qquad s(2n)=s(n), \; s(2n+1) = s(n)+s(n+1)\] (\seqnum{A002487} in the On-Line Encyclopedia of Integer Sequences (OEIS) \cite{OEIS}). The set of non-negative integers $n$ with the property $s(n) \equiv 0$ (mod $3$) is precisely the set of non-negative integers in the component of $G(2n, 8n + 5, 8n-5, 8n + 7, 8n-7)$ containing 0 \cite{ReznickSternArticle}. Some structural results about ECSDs also imply results about non-standard digital representations of integers, as we will see in Section \ref{components}.

Previous results about similar directed graphs have been explored by Tangjai, who in her 2014 Ph.D. thesis from the University of Illinois at Urbana-Champaign \cite{Tangjai} studied primarily $G(3n, 3n+1, 3n+5)$ on the non-negative integers. She focused on finding properties of the subset of $\N_0$ that were members of the component containing 0, looking at the asymptotic density and blocks of consecutive integers contained in this subset. She considered directed graphs only on the non-negative integers, which do not have the same general structure as when they are considered on all the integers. 

Questions of the density of components of ECSDs in the integers nonetheless provide an interesting further area of study. Some ECSDs clearly have stable asymptotic density; for example, each component of the ECSD $G(dn, dn+1, \ldots, dn+d-1)$ has asymptotic density $\frac12$ because one component contains all the positive integers and the other component contains all the negative integers. However, for $d\geq 4$, the asymptotic density of the component containing 0 in the ECSD $G(dn, dn+2, dn+3, \ldots, dn+d-1, dn+d+1)$ does not exist. 

In Section \ref{deg1}, we discuss the structure of ECSDs of degree 1, as they behave very differently from ECSDs of higher degree. In Section \ref{essentialthms}, we prove several theorems that hold for all ECSDs of degree at least 2. In particular, we prove that the number of components in such an ECSD must be finite, and every component must contain exactly one cycle. We also find graph isomorphisms between different ECSDs. 

In Section \ref{components}, we connect structural results about ECSDs with non-standard representations of integers. Recall standard base $d$, or $d$-ary, representations of natural numbers: every natural number can be written as a sum
\[\sum_{j=0}^k b_jd^j\text{ with } b_j \in \{a_1, \ldots, a_d\}\] 
for some $k \in \N$ if $\{a_1, \ldots, a_d\} = \{0, \ldots, d-1\}$. Generalizing this concept to non-standard representations, we choose non-standard representatives of each congruence class modulo $d$ for $\{a_1, \ldots, a_d\}$. We show that every integer can be written uniquely as the above sum
if and only if the ECSD $G(dn+a_1,\ldots, dn+a_d)$ has a single component and 0 is an element of the cycle.

\section{ECSDs of degree 1} \label{deg1}

\noindent
First we completely characterize ECSDs of degree 1, and we shall see that they behave quite differently from ECSDs of higher degree. 
Since the only exact covering system with a single congruence is $\{\modd{0} {1}\}$, the only possible representative sets are
\[S=\{(a,1)\} \text{  or  }S=\{(a,-1)\}.\]
Thus every ECSD of degree 1 is of the form $G(n+a)$ or $G(-n+a)$ for some $a \in \Z$.
The graphs of the form $G(n+a)$ with $a \neq 0$ are the only ECSDs where each component is acyclic, and $G(n)$ and $G(-n+a)$ are the only ECSDs with infinitely many components; we shall prove both claims in Section \ref{essentialthms}.

\begin{prop}
The graph $G(n+a)$, with $a \neq 0$, is the disjoint union of $|a|$ infinite paths.  When $a=0$, the graph is the disjoint union of infinitely many loops, one at each integer. 
\end{prop}

\begin{proof} Clearly $G(n)$ is the disjoint union of infinitely many loops, because each $n\in \Z$ is sent to itself (shown in Figure \ref{G(n)}).

\begin{figure}[h]
    \centering
    \begin{tikzpicture}[x=0.75pt,y=0.75pt,yscale=-1,xscale=1]

\draw [color={rgb, 255:red, 0; green, 0; blue, 0 }  ,draw opacity=1 ]   (169.35,85) .. controls (148.77,56.58) and (200.22,56.97) .. (187.23,82.41) ;
\draw [shift={(186.35,84)}, rotate = 300.65] [color={rgb, 255:red, 0; green, 0; blue, 0 }  ,draw opacity=1 ][line width=0.75]    (10.93,-3.29) .. controls (6.95,-1.4) and (3.31,-0.3) .. (0,0) .. controls (3.31,0.3) and (6.95,1.4) .. (10.93,3.29)   ;

\draw [color={rgb, 255:red, 0; green, 0; blue, 0 }  ,draw opacity=1 ]   (223.14,85) .. controls (202.56,56.58) and (254.01,56.97) .. (241.02,82.41) ;
\draw [shift={(240.14,84)}, rotate = 300.65] [color={rgb, 255:red, 0; green, 0; blue, 0 }  ,draw opacity=1 ][line width=0.75]    (10.93,-3.29) .. controls (6.95,-1.4) and (3.31,-0.3) .. (0,0) .. controls (3.31,0.3) and (6.95,1.4) .. (10.93,3.29)   ;

\draw [color={rgb, 255:red, 0; green, 0; blue, 0 }  ,draw opacity=1 ]   (275.5,85) .. controls (254.92,56.58) and (306.37,56.97) .. (293.38,82.41) ;
\draw [shift={(292.5,84)}, rotate = 300.65] [color={rgb, 255:red, 0; green, 0; blue, 0 }  ,draw opacity=1 ][line width=0.75]    (10.93,-3.29) .. controls (6.95,-1.4) and (3.31,-0.3) .. (0,0) .. controls (3.31,0.3) and (6.95,1.4) .. (10.93,3.29)   ;

\draw [color={rgb, 255:red, 0; green, 0; blue, 0 }  ,draw opacity=1 ]   (323.5,85) .. controls (302.92,56.58) and (354.37,56.97) .. (341.38,82.41) ;
\draw [shift={(340.5,84)}, rotate = 300.65] [color={rgb, 255:red, 0; green, 0; blue, 0 }  ,draw opacity=1 ][line width=0.75]    (10.93,-3.29) .. controls (6.95,-1.4) and (3.31,-0.3) .. (0,0) .. controls (3.31,0.3) and (6.95,1.4) .. (10.93,3.29)   ;

\draw [color={rgb, 255:red, 0; green, 0; blue, 0 }  ,draw opacity=1 ]   (371.5,85) .. controls (350.92,56.58) and (402.37,56.97) .. (389.38,82.41) ;
\draw [shift={(388.5,84)}, rotate = 300.65] [color={rgb, 255:red, 0; green, 0; blue, 0 }  ,draw opacity=1 ][line width=0.75]    (10.93,-3.29) .. controls (6.95,-1.4) and (3.31,-0.3) .. (0,0) .. controls (3.31,0.3) and (6.95,1.4) .. (10.93,3.29)   ;

\draw [color={rgb, 255:red, 0; green, 0; blue, 0 }  ,draw opacity=1 ]   (417.5,85) .. controls (396.92,56.58) and (448.37,56.97) .. (435.38,82.41) ;
\draw [shift={(434.5,84)}, rotate = 300.65] [color={rgb, 255:red, 0; green, 0; blue, 0 }  ,draw opacity=1 ][line width=0.75]    (10.93,-3.29) .. controls (6.95,-1.4) and (3.31,-0.3) .. (0,0) .. controls (3.31,0.3) and (6.95,1.4) .. (10.93,3.29)   ;

\draw [color={rgb, 255:red, 0; green, 0; blue, 0 }  ,draw opacity=1 ]   (465.5,85) .. controls (444.92,56.58) and (496.37,56.97) .. (483.38,82.41) ;
\draw [shift={(482.5,84)}, rotate = 300.65] [color={rgb, 255:red, 0; green, 0; blue, 0 }  ,draw opacity=1 ][line width=0.75]    (10.93,-3.29) .. controls (6.95,-1.4) and (3.31,-0.3) .. (0,0) .. controls (3.31,0.3) and (6.95,1.4) .. (10.93,3.29)   ;

\draw (379.04,94) node    {$1$};

\draw (425.04,94) node    {$2$};

\draw (331.04,94) node    {$0$};

\draw (283.04,94) node    {$-1$};

\draw (230.68,94) node    {$-2$};

\draw (176.9,94) node    {$-3$};

\draw (473.04,94) node    {$3$};

\draw (132,94) node    {$\cdots $};

\draw (513,94) node    {$\cdots $};

\end{tikzpicture}
    \caption{The ECSD $G(n)$.}
    \label{G(n)}
\end{figure}
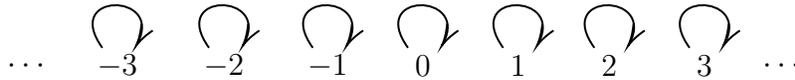

The graph $G(n+a)$ sends $n \in \Z$ to $n+a$, so the component containing $n$ is an infinite path where every element of the path is congruent to $n$ (mod $a$). Thus, there are $|a|$ disjoint infinite paths $P_i = \{i + ka \ : \ k \in \Z\}$ for $0 \leq i < |a|$. One path is shown in Figure \ref{G(n+a)}. \qedhere

\begin{figure}[h]
    \centering
    \begin{tikzpicture}[x=0.75pt,y=0.75pt,yscale=-1,xscale=1]

\draw (329.44,114) node    {$a$};

\draw (236.58,114) node    {$0$};

\draw (140.22,114) node    {$-a$};

\draw (424.8,114) node    {$2a$};

\draw (102,114) node    {$\cdots $};

\draw (190.15,114) node    {$\cdots $};

\draw (283.01,114) node    {$\cdots $};

\draw (375.87,114) node    {$\cdots $};

\draw (462,114) node    {$\cdots $};

\draw    (153.21,105.97) .. controls (176.2,88.73) and (200.31,89) .. (225.53,106.77) ;
\draw [shift={(227.07,107.88)}, rotate = 216.19] [color={rgb, 255:red, 0; green, 0; blue, 0 }  ][line width=0.75]    (10.93,-3.29) .. controls (6.95,-1.4) and (3.31,-0.3) .. (0,0) .. controls (3.31,0.3) and (6.95,1.4) .. (10.93,3.29)   ;

\draw    (246.08,107.91) .. controls (270.21,90.22) and (294.34,89.87) .. (318.46,106.85) ;
\draw [shift={(319.93,107.91)}, rotate = 216.25] [color={rgb, 255:red, 0; green, 0; blue, 0 }  ][line width=0.75]    (10.93,-3.29) .. controls (6.95,-1.4) and (3.31,-0.3) .. (0,0) .. controls (3.31,0.3) and (6.95,1.4) .. (10.93,3.29)   ;

\draw    (338.93,107.89) .. controls (363.88,89.63) and (388.02,88.83) .. (411.37,105.46) ;
\draw [shift={(412.8,106.51)}, rotate = 216.66] [color={rgb, 255:red, 0; green, 0; blue, 0 }  ][line width=0.75]    (10.93,-3.29) .. controls (6.95,-1.4) and (3.31,-0.3) .. (0,0) .. controls (3.31,0.3) and (6.95,1.4) .. (10.93,3.29)   ;

\end{tikzpicture}
    \caption{The ECSD $G(n+a)$.}
    \label{G(n+a)}
\end{figure}
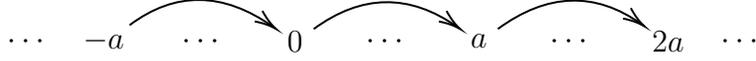
\end{proof}

\begin{prop}
The graph $G(-n+a)$ is a disjoint union of infinitely many 2-cycles covering $\Z$, with one exception being that, if $a$ is even, the graph has a loop at $\frac{1}{2}a$. If $a$ is odd, there is no such loop. Thus there are exactly two isomorphism classes of ECSDs of this type. 
\end{prop}

\begin{proof} Since $n \in \Z$ is sent to $-n+a$, and $-n+a$ is sent to $-(-n+a)+a = n$, there are infinitely many disjoint 2-cycles. If $a$ is even, then $n=-n+a$ for $n = \frac{a}{2}$. Thus, if $a$ is even there is a loop at $\frac{a}{2}$, and if $a$ is odd there is no such loop (see Figure \ref{G(-n+a)}).
\end{proof}

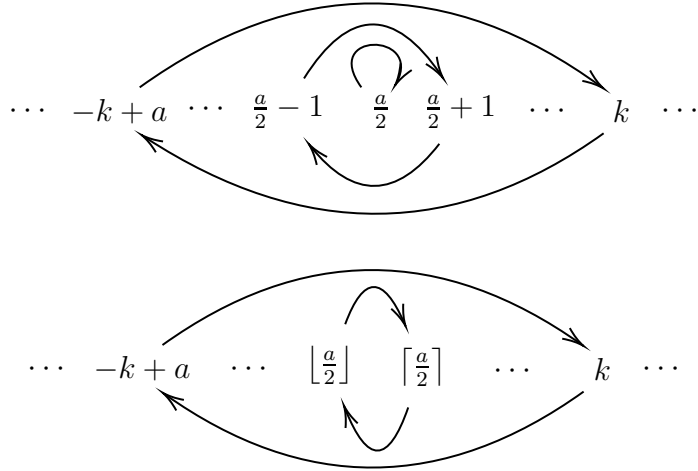
\begin{figure}[h]
    \centering
    \begin{tikzpicture}[x=0.75pt,y=0.75pt,yscale=-1,xscale=1]
    
\draw [color={rgb, 255:red, 0; green, 0; blue, 0 }  ,draw opacity=1 ]   (300.14,107) .. controls (279.56,78.58) and (331.01,78.97) .. (318.02,104.41) ;
\draw [shift={(317.14,106)}, rotate = 300.65] [color={rgb, 255:red, 0; green, 0; blue, 0 }  ,draw opacity=1 ][line width=0.75]    (10.93,-3.29) .. controls (6.95,-1.4) and (3.31,-0.3) .. (0,0) .. controls (3.31,0.3) and (6.95,1.4) .. (10.93,3.29)   ;

\draw (310,119) node    {$\frac{a}{2}$};

\draw (178.17,119) node    {$-k+a$};

\draw (133,119) node    {$\cdots $};

\draw (223,118) node    {$\cdots $};

\draw (431.31,119) node    {$k$};

\draw (394.98,119) node    {$\cdots $};

\draw (462,119) node    {$\cdots $};

\draw (350,119) node    {$\frac{a}{2} +1$};

\draw (263,119) node    {$\frac{a}{2} -1$};

\draw    (271.21,102.5) .. controls (293.73,66.71) and (316.78,66.18) .. (340.36,100.89) ;
\draw [shift={(341.44,102.5)}, rotate = 236.6] [color={rgb, 255:red, 0; green, 0; blue, 0 }  ][line width=0.75]    (10.93,-3.29) .. controls (6.95,-1.4) and (3.31,-0.3) .. (0,0) .. controls (3.31,0.3) and (6.95,1.4) .. (10.93,3.29)   ;

\draw    (339.46,135.5) .. controls (317.82,163.41) and (296.17,163.82) .. (274.53,136.76) ;
\draw [shift={(273.54,135.5)}, rotate = 412.21000000000004] [color={rgb, 255:red, 0; green, 0; blue, 0 }  ][line width=0.75]    (10.93,-3.29) .. controls (6.95,-1.4) and (3.31,-0.3) .. (0,0) .. controls (3.31,0.3) and (6.95,1.4) .. (10.93,3.29)   ;

\draw    (188.34,107) .. controls (266.79,50.61) and (344.39,50.53) .. (421.16,106.75) ;
\draw [shift={(422.32,107.6)}, rotate = 216.5] [color={rgb, 255:red, 0; green, 0; blue, 0 }  ][line width=0.75]    (10.93,-3.29) .. controls (6.95,-1.4) and (3.31,-0.3) .. (0,0) .. controls (3.31,0.3) and (6.95,1.4) .. (10.93,3.29)   ;

\draw    (191.61,132.58) .. controls (271.14,184.32) and (348.04,183.52) .. (422.32,130.19) ;
\draw [shift={(189.2,131)}, rotate = 33.56] [color={rgb, 255:red, 0; green, 0; blue, 0 }  ][line width=0.75]    (10.93,-3.29) .. controls (6.95,-1.4) and (3.31,-0.3) .. (0,0) .. controls (3.31,0.3) and (6.95,1.4) .. (10.93,3.29)   ;

\end{tikzpicture}

    \begin{tikzpicture}[x=0.75pt,y=0.75pt,yscale=-1,xscale=1]

\draw (187.17,83) node    {$-k+a$};

\draw (419.31,83) node    {$k$};

\draw (281.65,81) node    {$\left\lfloor \frac{a}{2}\right\rfloor $};

\draw (328.33,82) node    {$\left\lceil \frac{a}{2}\right\rceil $};

\draw (140,83) node    {$\cdots $};

\draw (241.99,83) node    {$\cdots $};

\draw (374.98,84) node    {$\cdots $};

\draw (450,83) node    {$\cdots $};

\draw    (321.36,102.5) .. controls (311.42,131.27) and (300.83,131.53) .. (289.58,103.27) ;
\draw [shift={(288.9,101.5)}, rotate = 429.03999999999996] [color={rgb, 255:red, 0; green, 0; blue, 0 }  ][line width=0.75]    (10.93,-3.29) .. controls (6.95,-1.4) and (3.31,-0.3) .. (0,0) .. controls (3.31,0.3) and (6.95,1.4) .. (10.93,3.29)   ;

\draw    (289.57,60.5) .. controls (298.78,35.65) and (308.72,35.47) .. (319.39,59.97) ;
\draw [shift={(320.04,61.5)}, rotate = 247.26] [color={rgb, 255:red, 0; green, 0; blue, 0 }  ][line width=0.75]    (10.93,-3.29) .. controls (6.95,-1.4) and (3.31,-0.3) .. (0,0) .. controls (3.31,0.3) and (6.95,1.4) .. (10.93,3.29)   ;

\draw    (197.42,71) .. controls (267.62,20.64) and (338.06,20.56) .. (408.75,70.73) ;
\draw [shift={(409.82,71.49)}, rotate = 215.64] [color={rgb, 255:red, 0; green, 0; blue, 0 }  ][line width=0.75]    (10.93,-3.29) .. controls (6.95,-1.4) and (3.31,-0.3) .. (0,0) .. controls (3.31,0.3) and (6.95,1.4) .. (10.93,3.29)   ;

\draw    (409.82,94.51) .. controls (339.13,145.19) and (268.68,145.61) .. (198.48,95.76) ;
\draw [shift={(197.42,95)}, rotate = 395.65] [color={rgb, 255:red, 0; green, 0; blue, 0 }  ][line width=0.75]    (10.93,-3.29) .. controls (6.95,-1.4) and (3.31,-0.3) .. (0,0) .. controls (3.31,0.3) and (6.95,1.4) .. (10.93,3.29)   ;

\end{tikzpicture}
    \caption{The ECSDs $G(-n+a)$ with $a$ even above and $a$ odd below.}
    \label{G(-n+a)}
\end{figure}

In this simple case where the degree is 1, we can already see that the signs of the $d_i$ (in this case, determining $\pm n$) matter significantly to the structure of an ECSD. This will continue to hold for ECSDs of greater degree. 

\section{Basic theorems about ECSDs}
\label{essentialthms}

For an ECSD of degree $r \geq 2$, we will prove that each of finitely many components has exactly one cycle, and that each vertex of a cycle is the root of a tree with $r-1$ successors, each of which is the root of an infinite $r$-ary tree. In order to do so, we first require a lemma.

\begin{lem} \label{ancestorlemma} Given an ECSD, for each $n \in \Z$, there is an ancestor $m$ of $n$ such that $|P(m)| \geq |m|$. Moreover, for each ECSD of degree at least 2 there exists some $N \in \N$ such that if $m \in \Z$ and $|P(m)| \geq |m|$, then $|m| \leq N$.
\end{lem}

\begin{proof} Suppose $n \in \Z$ has no such ancestor. Then $|P^{k+1}(n)| < |P^k(n)| < \cdots < |n|$ for each $k \in \N$. But there are only finitely many nonnegative numbers less than $|n|$, contradicting the assumption.

Let $m \in \Z$ be such an ancestor. Suppose the edge $(P(m), m)$ comes from the pair $(a_i, d_i)$, i.e., when $m = P(m) d_i + a_i$ (recall that $d_i$ may be negative). We therefore have the inequality $\left|\frac{m-a_i}{d_i} \right| \geq |m|$, which simplifies to $0 \geq ((d_i-1)m+a_i)((d_i+1)m-a_i)$. 
This ensures that $m$ must be between $-\frac{a_i}{d_i-1}$ and $\frac{a_i}{d_i+1}$ (when $d_i \neq \pm 1$, which is true assuming $r \geq 2$).

Let $N_i = \max\left\{\left|-\frac{a_i}{d_i-1}\right|, \left|\frac{a_i}{d_i+1}\right|\right\}$. Then $|m| \leq N_i$. Now let $N = \max_{i}N_i$, so for any $m \in \Z$ such that $|P(m)| \geq |m|$, we have $|m| \leq N$. 
\end{proof}

 From this lemma, we can conclude that the number of components of an ECSD of degree at least 2 must be finite: if each $n \in \Z$ has an ancestor in $\{m \in \Z : |m| \leq N\}$, the number of components is at most $2N+1$. Note that this is a rather poor upper bound. Although it is simple to determine the number of components for a given ECSD computationally, the number of components of a general ECSD seems to often be irregular. In future work \cite{Neidinger} we find all ECSDs of degree 2 with a single component, but the ECSD $G(3n, 3n+1, 3n-a)$ has a single component whenever $a$ is in the sequence \seqnum{A110081} in the On-Line Encyclopedia of Integer Sequences (OEIS) \cite{OEIS}, which is infinite, irregular, and has density zero. In section \ref{components} we will show why single-component ECSDs are of special interest.

Instead, we will use this lemma to prove that each component has exactly one cycle:

\begin{thm} \label{componentcycles}
Each component of an ECSD of degree at least 2 has exactly one cycle. 
\end{thm}

\begin{proof} 
Let $n \in \Z$. By Lemma \ref{ancestorlemma}, $n$ has some ancestor $m_1$ such that $|P(m_1)| \geq |m_1|$. Similarly, $m_1$ has such an ancestor $m_2$. Thus, we can construct an infinite sequence of ancestors $m_1, m_2, \ldots, m_k, \ldots$ of $n$ such that $|m_k| \leq N$ for each $k \in \N$. Thus, this sequence contains repetitions, and $n$ is the descendant of some cycle.

Since every vertex has indegree 1, every edge adjacent to a cycle must be oriented away from the cycle. Also, every walk that intersects a cycle is oriented away from this cycle, and cannot intersect itself, or else would violate the indegree 1 condition. Thus, there is at most one cycle per component. 
\end{proof}

We begin to see that the cycles of ECSDs are essential to their overall structure, and indeed, proving results about the cycles of ECSDs is the main goal of our future efforts. To this end, we shall study the integers in cycles of an ECSD, which we will call \textit{cyclic vertices}. We will use the notation $C = (c_1, c_2, \ldots, c_k)$ to denote a cycle of length $k$ in an ECSD, where $P(c_{i+1}) = c_i$ for each $i$, and $P(c_1) = c_k$. 

\begin{obs} \label{ancestorcyclicvx}
If a vertex $n$ is an ancestor of a cyclic vertex, then $n$ is also a cyclic vertex.
\end{obs}

\begin{proof}
All edges adjacent to a cycle must be oriented away from the cycle, and thus any ancestor of a cyclic vertex must also be on the cycle.
\end{proof}

One way we can study the structure of ECSDs is by using graph isomorphisms to conclude that many different ECSDs must all have the same structure. In fact, two ECSDs of the same degree are isomorphic if they have the same number of components, and the same number of cyclic vertices in each corresponding component. Thus, the structure of an ECSD up to graph isomorphism can be represented by its degree along with a weakly increasing list of numbers of cyclic vertices in each component. For example, the ECSD in Figure \ref{1(2n,2n-9)} could be represented by the list $[2; 1,1,2,6]$.

There are two main types of isomorphisms of covering system digraphs, both of which are automorphisms of the integers. Specifically, the isomorphisms of covering system digraphs are automorphisms of $\Z$ that respect $\Z$ as a graph, that is, the images of neighbors $n+1$ and $n-1$ of $n \in \Z$ under the automorphism are still neighbors of the image of $n$. 
The first type of ECSD isomorphism results from a shift on the integers: 

\begin{prop} \label{firstgraphisomorphism}
For each $k \in \Z$ and representative sets of exact covering systems
\begin{align*}
S &= \{(a_i, d_i) \in \Z^2 : 1 \leq i \leq r\}, \\
S'&= \{(a_i-(d_i-1)k, d_i) \in \Z^2  : 1 \leq i \leq r\},
\end{align*} 
the map $\phi_{1,k}: \Z \to \Z$ defined by $\phi_{1,k}(n) = n+k$ is a graph isomorphism between $G_S$ and $G_{S'}$.
\end{prop}

\begin{proof} 
By definition, $\phi_{1,k}$ maps $V(G_S)$ to $V(G_{S'})$:
for each edge $(n, d_in + a_i)\in E(G_S)$, we have the edge $(\phi_{1,k}(n), \phi_{1,k}(d_i n+a_i)) = (n+k, d_in + a_i + k) \in E(G_{S'})$.
If we let $m = n+k$ and $n = m-k$, this edge is $(m, d_i m + a_i - (d_i-1)k) \in E(G_{S'})$.
\end{proof}

The second type of ECSD isomorphism results from a shift and flip on the integers:

\begin{prop}\label{secondgraphisomorphism}
For each $k \in \Z$ and representative sets of exact covering systems
\begin{align*}
S &= \{(a_i, d_i) \in \Z^2 : 1 \leq i \leq r\}, \\
S'&= \{(-a_i+(d_i-1)k, d_i) \in \Z^2 : 1 \leq i \leq r\},
\end{align*} 
the map $\phi_{2,k}: \Z \to \Z$ defined by $\phi_{2,k}(n) = -(n+k)$ is a graph isomorphism between $G_S$ and $G_{S'}$.
\end{prop}

\begin{proof}
By definition, $\phi_{2,k}$ maps $V(G_S)$ to $V(G_{S'})$: 
for each edge $(n, d_i n + a_i)\in E(G_S)$, we have the edge $(\phi_{2,k}(n), \phi_{2,k}(d_i n+a_i)) = (-(n+k), -(d_i n + a_i + k)) \in E(G_{S'})$.
If we let $m= -(n+k)$ and $n = -m-k$, 
this edge is $(m, d_i m - a_i + (d_i-1)k) \in E(G_{S'})$.
\end{proof}

Note that in both of the above propositions, $S$ is an exact covering system if and only if $S'$ is also an exact covering system: for example, in Proposition \ref{firstgraphisomorphism}, if all $n \in \Z$ are covered exactly once by $S$ then all $n+k \in \Z$ are covered exactly once by $S'$.

In forthcoming work \cite{Neidinger}, we shall use these isomorphisms to classify all ECSDs of degree 2. We split the degree 2 ECSDs into three types based on the signs of their $d_i$: $G(2n+a_1, 2n+a_2)$, $G(-2n+a_1, 2n+a_2)$, and $G(-2n+a_1,-2n+a_2)$. We then use these isomorphisms to simplify each type and determine the vertices contained in cycles for each. 
For now, we turn our attention to single-component ECSDs.

\section{Single-component ECSDs and non-standard digital representations \label{components}}

In an ECSD, if $d_1=d_2=\cdots=d_r$, then its structure is closely linked to digital representations. 

Consider an ECSD $G_S$ with representative set $S = \{(a_i, d) \in \Z^2 : 1 \leq i \leq d\}$, where $d$ is a positive integer. There are $d$ pairs, each representing a different congruence class modulo $d$.
Then the set
\begin{equation} \label{k-descendant positive}
\left\{\sum_{j=0}^k b_jd^j : b_j \in \{a_1, \ldots, a_d\} \right\}
\end{equation}
is the set of all $(k+1)$-descendants of $0$ in $G_S$: each coefficient $b_j$ is a representative $a_i$ determined by the choice of each corresponding branch in the path from $0$ to the descendant.

This is more clearly illustrated in Figure \ref{0-descendants}, which shows the 1, 2, and 3-descendants of 0 in the ECSD $G(2n+a_1, 2n+a_2)$. Here, the representatives of the two binary congruence classes are $a_1$ and $a_2$. For example, if we travel from 0 through an edge $(n, 2n+a_1)$, then two edges $(n, 2n+a_2)$, we arrive at the 3-descendant 
\[2(2(2(0)+a_1)+a_2)+a_2 = 2^2a_1 + 2^1a_2 + 2^0a_2 = 4a_1 + 2a_2 + a_2.\]

\begin{figure}[h]
\centering
\begin{tikzpicture}[scale=.7]
    \small
    \centering
        \node (1) at (0,0) {$0$};
        \node (2) at (-4,-1) {$a_1$};
        \node (3) at (4,-1) {$a_2$};
        \node (4) at (-6, -2) {$2a_1+a_1$};
        \node (5) at ( -2, -2) {$2a_1+a_2$};
        \node (6) at (2,-2) {$2a_2+a_1$};
        \node (7) at (6, -2) {$2a_2+a_2$};
        
        \node (8) at (-7.5, -4) {$4a_1+2a_1+a_1$};
        \node (9) at (-5.2, -5) {$4a_1+2a_1+a_2$};
        \node (10) at (-3.5, -4) {$4a_1+2a_2+a_1$};
        \node (11) at (-1.2, -5) {$4a_1+2a_2+a_2$};
        \node (12) at (.5, -4) {$4a_2+2a_1+a_1$};
        \node (13) at (2.8, -5) {$4a_2+2a_1+a_2$};
        \node (14) at (4.5, -4) {$4a_2+2a_2+a_1$};
        \node (15) at (6.8, -5) {$4a_2+2a_2+a_2$};

        \edge{1}{2}
        \edge{1}{3}
        \edge{2}{4}
        \edge{2}{5}
        \edge{3}{6}
        \edge{3}{7}
        
        \edge{4}{8}
        \edge{4}{9}
        \edge{5}{10}
        \edge{5}{11}
        \edge{6}{12}
        \edge{6}{13}
        \edge{7}{14}
        \edge{7}{15}
        
\end{tikzpicture}
\caption{The 1, 2, and 3-descendants of 0 in $G(2n+a_1, 2n+a_2)$, with vertices possibly non-unique because of cycles.}
\label{0-descendants}
\end{figure}
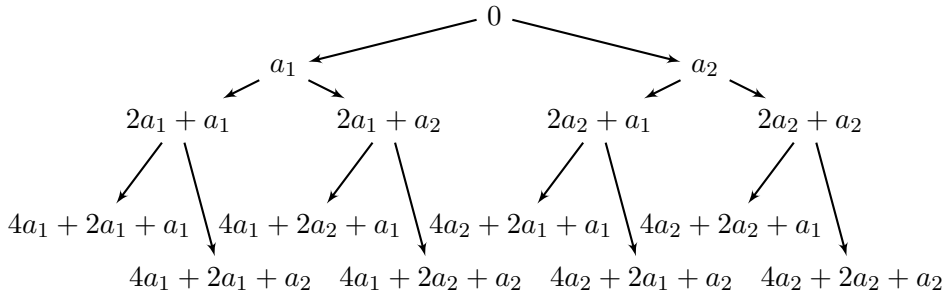

If we choose $a_1=0$ and $a_2=1$, this representation becomes the standard binary representation. The corresponding ECSD $G(2n, 2n+1)$ is shown in Figure \ref{binary}, with $(n, 2n)$ edges colored red and $(n, 2n+1)$ edges colored blue. In this case, there are two loops and two components. Here, we can see that the number 6 can be reached via two blue steps then a red step from 0. The standard binary representation of 6 is $2^2 \cdot 1 + 2^1 \cdot 1 + 2^0 \cdot 0$, written as ``110" in digital form. 

\begin{figure}[h] 
\centering
\begin{tikzpicture}[x=0.75pt,y=0.75pt,yscale=-1,xscale=1]

\draw [color={rgb, 255:red, 255; green, 0; blue, 31 }  ,draw opacity=1 ]   (73.5,26) .. controls (52.92,-2.42) and (104.37,-2.03) .. (91.38,23.41) ;
\draw [shift={(90.5,25)}, rotate = 300.65] [color={rgb, 255:red, 255; green, 0; blue, 31 }  ,draw opacity=1 ][line width=0.75]    (10.93,-3.29) .. controls (6.95,-1.4) and (3.31,-0.3) .. (0,0) .. controls (3.31,0.3) and (6.95,1.4) .. (10.93,3.29)   ;

\draw [color={rgb, 255:red, 0; green, 115; blue, 255 }  ,draw opacity=1 ]   (263,25.5) .. controls (242.42,-2.92) and (293.87,-2.53) .. (280.88,22.91) ;
\draw [shift={(280,24.5)}, rotate = 300.65] [color={rgb, 255:red, 0; green, 115; blue, 255 }  ,draw opacity=1 ][line width=0.75]    (10.93,-3.29) .. controls (6.95,-1.4) and (3.31,-0.3) .. (0,0) .. controls (3.31,0.3) and (6.95,1.4) .. (10.93,3.29)   ;

\draw (83,34) node    {$0$};

\draw (89,73) node [anchor=east] [inner sep=0.75pt]    {$1$};

\draw (42,114) node    {$2$};

\draw (122.5,114) node    {$3$};

\draw (22,153) node    {$4$};

\draw (62.5,153) node    {$5$};

\draw (102.5,153) node    {$6$};

\draw (142.5,153) node    {$7$};

\draw (11,192.5) node    {$\vdots $};

\draw (31,192.5) node    {$\vdots $};

\draw (51,193) node    {$\vdots $};

\draw (71.5,192.5) node    {$\vdots $};

\draw (93.5,193) node    {$\vdots $};

\draw (113.5,193) node    {$\vdots $};

\draw (133.5,193) node    {$\vdots $};

\draw (153.5,192.5) node    {$\vdots $};

\draw (272,34) node    {$-1$};

\draw (282.5,73) node [anchor=east] [inner sep=0.75pt]    {$-2$};

\draw (231.5,114) node    {$-3$};

\draw (312,114) node    {$-4$};

\draw (212,153) node    {$-5$};

\draw (252,153) node    {$-6$};

\draw (292,152.5) node    {$-7$};

\draw (332,153) node    {$-8$};

\draw (202.5,192.5) node    {$\vdots $};

\draw (220.5,192.5) node    {$\vdots $};

\draw (240.5,192.5) node    {$\vdots $};

\draw (261,192.5) node    {$\vdots $};

\draw (283,192.5) node    {$\vdots $};

\draw (303,192.5) node    {$\vdots $};

\draw (323,192.5) node    {$\vdots $};

\draw (343,192.5) node    {$\vdots $};

\draw [color={rgb, 255:red, 0; green, 115; blue, 255 }  ,draw opacity=1 ]   (83,46) -- (83,59) ;
\draw [shift={(83,61)}, rotate = 270] [color={rgb, 255:red, 0; green, 115; blue, 255 }  ,draw opacity=1 ][line width=0.75]    (10.93,-3.29) .. controls (6.95,-1.4) and (3.31,-0.3) .. (0,0) .. controls (3.31,0.3) and (6.95,1.4) .. (10.93,3.29)   ;

\draw [color={rgb, 255:red, 0; green, 115; blue, 255 }  ,draw opacity=1 ]   (92,82.34) -- (112.11,103.22) ;
\draw [shift={(113.5,104.66)}, rotate = 226.07] [color={rgb, 255:red, 0; green, 115; blue, 255 }  ,draw opacity=1 ][line width=0.75]    (10.93,-3.29) .. controls (6.95,-1.4) and (3.31,-0.3) .. (0,0) .. controls (3.31,0.3) and (6.95,1.4) .. (10.93,3.29)   ;

\draw [color={rgb, 255:red, 255; green, 0; blue, 31 }  ,draw opacity=1 ]   (74,82) -- (52.41,103.59) ;
\draw [shift={(51,105)}, rotate = 315] [color={rgb, 255:red, 255; green, 0; blue, 31 }  ,draw opacity=1 ][line width=0.75]    (10.93,-3.29) .. controls (6.95,-1.4) and (3.31,-0.3) .. (0,0) .. controls (3.31,0.3) and (6.95,1.4) .. (10.93,3.29)   ;

\draw [color={rgb, 255:red, 255; green, 0; blue, 31 }  ,draw opacity=1 ]   (35.85,126) -- (29.07,139.22) ;
\draw [shift={(28.15,141)}, rotate = 297.15] [color={rgb, 255:red, 255; green, 0; blue, 31 }  ,draw opacity=1 ][line width=0.75]    (10.93,-3.29) .. controls (6.95,-1.4) and (3.31,-0.3) .. (0,0) .. controls (3.31,0.3) and (6.95,1.4) .. (10.93,3.29)   ;

\draw [color={rgb, 255:red, 0; green, 115; blue, 255 }  ,draw opacity=1 ]   (48.31,126) -- (55.26,139.23) ;
\draw [shift={(56.19,141)}, rotate = 242.27] [color={rgb, 255:red, 0; green, 115; blue, 255 }  ,draw opacity=1 ][line width=0.75]    (10.93,-3.29) .. controls (6.95,-1.4) and (3.31,-0.3) .. (0,0) .. controls (3.31,0.3) and (6.95,1.4) .. (10.93,3.29)   ;

\draw [color={rgb, 255:red, 255; green, 0; blue, 31 }  ,draw opacity=1 ]   (116.35,126) -- (109.57,139.22) ;
\draw [shift={(108.65,141)}, rotate = 297.15] [color={rgb, 255:red, 255; green, 0; blue, 31 }  ,draw opacity=1 ][line width=0.75]    (10.93,-3.29) .. controls (6.95,-1.4) and (3.31,-0.3) .. (0,0) .. controls (3.31,0.3) and (6.95,1.4) .. (10.93,3.29)   ;

\draw [color={rgb, 255:red, 0; green, 115; blue, 255 }  ,draw opacity=1 ]   (128.65,126) -- (135.43,139.22) ;
\draw [shift={(136.35,141)}, rotate = 242.85] [color={rgb, 255:red, 0; green, 115; blue, 255 }  ,draw opacity=1 ][line width=0.75]    (10.93,-3.29) .. controls (6.95,-1.4) and (3.31,-0.3) .. (0,0) .. controls (3.31,0.3) and (6.95,1.4) .. (10.93,3.29)   ;

\draw [color={rgb, 255:red, 255; green, 0; blue, 31 }  ,draw opacity=1 ]   (18.66,165) -- (14.88,178.57) ;
\draw [shift={(14.34,180.5)}, rotate = 285.56] [color={rgb, 255:red, 255; green, 0; blue, 31 }  ,draw opacity=1 ][line width=0.75]    (10.93,-3.29) .. controls (6.95,-1.4) and (3.31,-0.3) .. (0,0) .. controls (3.31,0.3) and (6.95,1.4) .. (10.93,3.29)   ;

\draw [color={rgb, 255:red, 0; green, 115; blue, 255 }  ,draw opacity=1 ]   (24.73,165) -- (27.82,178.55) ;
\draw [shift={(28.27,180.5)}, rotate = 257.15999999999997] [color={rgb, 255:red, 0; green, 115; blue, 255 }  ,draw opacity=1 ][line width=0.75]    (10.93,-3.29) .. controls (6.95,-1.4) and (3.31,-0.3) .. (0,0) .. controls (3.31,0.3) and (6.95,1.4) .. (10.93,3.29)   ;

\draw [color={rgb, 255:red, 255; green, 0; blue, 31 }  ,draw opacity=1 ]   (59.05,165) -- (55,179.08) ;
\draw [shift={(54.45,181)}, rotate = 286.04] [color={rgb, 255:red, 255; green, 0; blue, 31 }  ,draw opacity=1 ][line width=0.75]    (10.93,-3.29) .. controls (6.95,-1.4) and (3.31,-0.3) .. (0,0) .. controls (3.31,0.3) and (6.95,1.4) .. (10.93,3.29)   ;

\draw [color={rgb, 255:red, 0; green, 115; blue, 255 }  ,draw opacity=1 ]   (65.23,165) -- (68.32,178.55) ;
\draw [shift={(68.77,180.5)}, rotate = 257.15999999999997] [color={rgb, 255:red, 0; green, 115; blue, 255 }  ,draw opacity=1 ][line width=0.75]    (10.93,-3.29) .. controls (6.95,-1.4) and (3.31,-0.3) .. (0,0) .. controls (3.31,0.3) and (6.95,1.4) .. (10.93,3.29)   ;

\draw [color={rgb, 255:red, 255; green, 0; blue, 31 }  ,draw opacity=1 ]   (99.8,165) -- (96.64,179.05) ;
\draw [shift={(96.2,181)}, rotate = 282.68] [color={rgb, 255:red, 255; green, 0; blue, 31 }  ,draw opacity=1 ][line width=0.75]    (10.93,-3.29) .. controls (6.95,-1.4) and (3.31,-0.3) .. (0,0) .. controls (3.31,0.3) and (6.95,1.4) .. (10.93,3.29)   ;

\draw [color={rgb, 255:red, 0; green, 115; blue, 255 }  ,draw opacity=1 ]   (105.8,165) -- (109.67,179.07) ;
\draw [shift={(110.2,181)}, rotate = 254.62] [color={rgb, 255:red, 0; green, 115; blue, 255 }  ,draw opacity=1 ][line width=0.75]    (10.93,-3.29) .. controls (6.95,-1.4) and (3.31,-0.3) .. (0,0) .. controls (3.31,0.3) and (6.95,1.4) .. (10.93,3.29)   ;

\draw [color={rgb, 255:red, 255; green, 0; blue, 31 }  ,draw opacity=1 ]   (139.8,165) -- (136.64,179.05) ;
\draw [shift={(136.2,181)}, rotate = 282.68] [color={rgb, 255:red, 255; green, 0; blue, 31 }  ,draw opacity=1 ][line width=0.75]    (10.93,-3.29) .. controls (6.95,-1.4) and (3.31,-0.3) .. (0,0) .. controls (3.31,0.3) and (6.95,1.4) .. (10.93,3.29)   ;

\draw [color={rgb, 255:red, 0; green, 115; blue, 255 }  ,draw opacity=1 ]   (145.84,165) -- (149.62,178.57) ;
\draw [shift={(150.16,180.5)}, rotate = 254.44] [color={rgb, 255:red, 0; green, 115; blue, 255 }  ,draw opacity=1 ][line width=0.75]    (10.93,-3.29) .. controls (6.95,-1.4) and (3.31,-0.3) .. (0,0) .. controls (3.31,0.3) and (6.95,1.4) .. (10.93,3.29)   ;

\draw [color={rgb, 255:red, 255; green, 0; blue, 31 }  ,draw opacity=1 ]   (272,46) -- (272,59) ;
\draw [shift={(272,61)}, rotate = 270] [color={rgb, 255:red, 255; green, 0; blue, 31 }  ,draw opacity=1 ][line width=0.75]    (10.93,-3.29) .. controls (6.95,-1.4) and (3.31,-0.3) .. (0,0) .. controls (3.31,0.3) and (6.95,1.4) .. (10.93,3.29)   ;

\draw [color={rgb, 255:red, 255; green, 0; blue, 31 }  ,draw opacity=1 ]   (283.71,85) -- (298.9,100.57) ;
\draw [shift={(300.29,102)}, rotate = 225.71] [color={rgb, 255:red, 255; green, 0; blue, 31 }  ,draw opacity=1 ][line width=0.75]    (10.93,-3.29) .. controls (6.95,-1.4) and (3.31,-0.3) .. (0,0) .. controls (3.31,0.3) and (6.95,1.4) .. (10.93,3.29)   ;

\draw [color={rgb, 255:red, 0; green, 115; blue, 255 }  ,draw opacity=1 ]   (260.15,85) -- (244.76,100.58) ;
\draw [shift={(243.35,102)}, rotate = 314.65] [color={rgb, 255:red, 0; green, 115; blue, 255 }  ,draw opacity=1 ][line width=0.75]    (10.93,-3.29) .. controls (6.95,-1.4) and (3.31,-0.3) .. (0,0) .. controls (3.31,0.3) and (6.95,1.4) .. (10.93,3.29)   ;

\draw [color={rgb, 255:red, 0; green, 115; blue, 255 }  ,draw opacity=1 ]   (225.5,126) -- (218.89,139.21) ;
\draw [shift={(218,141)}, rotate = 296.57] [color={rgb, 255:red, 0; green, 115; blue, 255 }  ,draw opacity=1 ][line width=0.75]    (10.93,-3.29) .. controls (6.95,-1.4) and (3.31,-0.3) .. (0,0) .. controls (3.31,0.3) and (6.95,1.4) .. (10.93,3.29)   ;

\draw [color={rgb, 255:red, 255; green, 0; blue, 31 }  ,draw opacity=1 ]   (237.81,126) -- (244.76,139.23) ;
\draw [shift={(245.69,141)}, rotate = 242.27] [color={rgb, 255:red, 255; green, 0; blue, 31 }  ,draw opacity=1 ][line width=0.75]    (10.93,-3.29) .. controls (6.95,-1.4) and (3.31,-0.3) .. (0,0) .. controls (3.31,0.3) and (6.95,1.4) .. (10.93,3.29)   ;

\draw [color={rgb, 255:red, 0; green, 115; blue, 255 }  ,draw opacity=1 ]   (305.77,126) -- (299.16,138.73) ;
\draw [shift={(298.23,140.5)}, rotate = 297.45] [color={rgb, 255:red, 0; green, 115; blue, 255 }  ,draw opacity=1 ][line width=0.75]    (10.93,-3.29) .. controls (6.95,-1.4) and (3.31,-0.3) .. (0,0) .. controls (3.31,0.3) and (6.95,1.4) .. (10.93,3.29)   ;

\draw [color={rgb, 255:red, 255; green, 0; blue, 31 }  ,draw opacity=1 ]   (318.15,126) -- (324.93,139.22) ;
\draw [shift={(325.85,141)}, rotate = 242.85] [color={rgb, 255:red, 255; green, 0; blue, 31 }  ,draw opacity=1 ][line width=0.75]    (10.93,-3.29) .. controls (6.95,-1.4) and (3.31,-0.3) .. (0,0) .. controls (3.31,0.3) and (6.95,1.4) .. (10.93,3.29)   ;

\draw [color={rgb, 255:red, 0; green, 115; blue, 255 }  ,draw opacity=1 ]   (209.11,165) -- (205.85,178.56) ;
\draw [shift={(205.39,180.5)}, rotate = 283.52] [color={rgb, 255:red, 0; green, 115; blue, 255 }  ,draw opacity=1 ][line width=0.75]    (10.93,-3.29) .. controls (6.95,-1.4) and (3.31,-0.3) .. (0,0) .. controls (3.31,0.3) and (6.95,1.4) .. (10.93,3.29)   ;

\draw [color={rgb, 255:red, 255; green, 0; blue, 31 }  ,draw opacity=1 ]   (214.58,165) -- (217.5,178.54) ;
\draw [shift={(217.92,180.5)}, rotate = 257.86] [color={rgb, 255:red, 255; green, 0; blue, 31 }  ,draw opacity=1 ][line width=0.75]    (10.93,-3.29) .. controls (6.95,-1.4) and (3.31,-0.3) .. (0,0) .. controls (3.31,0.3) and (6.95,1.4) .. (10.93,3.29)   ;

\draw [color={rgb, 255:red, 0; green, 115; blue, 255 }  ,draw opacity=1 ]   (248.51,165) -- (244.55,178.58) ;
\draw [shift={(243.99,180.5)}, rotate = 286.23] [color={rgb, 255:red, 0; green, 115; blue, 255 }  ,draw opacity=1 ][line width=0.75]    (10.93,-3.29) .. controls (6.95,-1.4) and (3.31,-0.3) .. (0,0) .. controls (3.31,0.3) and (6.95,1.4) .. (10.93,3.29)   ;

\draw [color={rgb, 255:red, 255; green, 0; blue, 31 }  ,draw opacity=1 ]   (254.73,165) -- (257.82,178.55) ;
\draw [shift={(258.27,180.5)}, rotate = 257.15999999999997] [color={rgb, 255:red, 255; green, 0; blue, 31 }  ,draw opacity=1 ][line width=0.75]    (10.93,-3.29) .. controls (6.95,-1.4) and (3.31,-0.3) .. (0,0) .. controls (3.31,0.3) and (6.95,1.4) .. (10.93,3.29)   ;

\draw [color={rgb, 255:red, 0; green, 115; blue, 255 }  ,draw opacity=1 ]   (289.3,164.5) -- (286.14,178.55) ;
\draw [shift={(285.7,180.5)}, rotate = 282.68] [color={rgb, 255:red, 0; green, 115; blue, 255 }  ,draw opacity=1 ][line width=0.75]    (10.93,-3.29) .. controls (6.95,-1.4) and (3.31,-0.3) .. (0,0) .. controls (3.31,0.3) and (6.95,1.4) .. (10.93,3.29)   ;

\draw [color={rgb, 255:red, 255; green, 0; blue, 31 }  ,draw opacity=1 ]   (295.3,164.5) -- (299.17,178.57) ;
\draw [shift={(299.7,180.5)}, rotate = 254.62] [color={rgb, 255:red, 255; green, 0; blue, 31 }  ,draw opacity=1 ][line width=0.75]    (10.93,-3.29) .. controls (6.95,-1.4) and (3.31,-0.3) .. (0,0) .. controls (3.31,0.3) and (6.95,1.4) .. (10.93,3.29)   ;

\draw [color={rgb, 255:red, 0; green, 115; blue, 255 }  ,draw opacity=1 ]   (329.27,165) -- (326.18,178.55) ;
\draw [shift={(325.73,180.5)}, rotate = 282.84000000000003] [color={rgb, 255:red, 0; green, 115; blue, 255 }  ,draw opacity=1 ][line width=0.75]    (10.93,-3.29) .. controls (6.95,-1.4) and (3.31,-0.3) .. (0,0) .. controls (3.31,0.3) and (6.95,1.4) .. (10.93,3.29)   ;

\draw [color={rgb, 255:red, 255; green, 0; blue, 31 }  ,draw opacity=1 ]   (335.34,165) -- (339.12,178.57) ;
\draw [shift={(339.66,180.5)}, rotate = 254.44] [color={rgb, 255:red, 255; green, 0; blue, 31 }  ,draw opacity=1 ][line width=0.75]    (10.93,-3.29) .. controls (6.95,-1.4) and (3.31,-0.3) .. (0,0) .. controls (3.31,0.3) and (6.95,1.4) .. (10.93,3.29)   ;

\end{tikzpicture}
\caption{The ECSD $G(2n, 2n+1)$ with $(n, 2n)$ edges colored red and $(n, 2n+1)$ edges colored blue.}
\label{binary}
\end{figure}
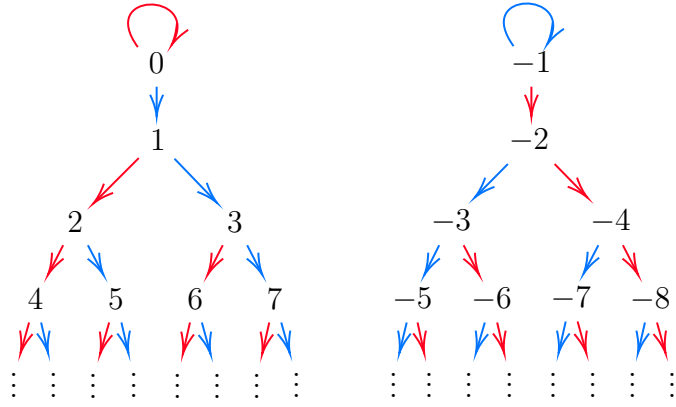

Thus, one can see that the component of $G_S$ containing 0, assuming that 0 is a cyclic vertex, is the set of integers which for some $k \in \N$ can be expressed as the sum in Equation \ref{k-descendant positive} (as every vertex $m$ in the component containing 0 will have a walk from $0$ to $m$). For the standard binary representation, we can see from Figure \ref{binary} that every non-negative integer has such a representation. The same is true for any standard $d$-ary representation (i.e., with representatives $0,1,\ldots,d-1$). 

Through non-standard choices of representatives, we are able to create different non-standard representations. Moreover, we have not restricted ourselves to positive moduli, so we may also consider ECSDs with modulus representative $-d$ for some integer $d >0$, i.e., the ECSD $G_{S'}$ with $S' = \{(a_i, -d) \in \Z^2 : 1 \leq i \leq d\}$. In this case, the $(k+1)$-descendants of 0 in $G_{S'}$ are

\begin{equation} \label{k-descendants negative}
\left\{\sum_{j=0}^k b_j(-d)^j : b_j \in \{a_1, \ldots, a_d\}\right\}.
\end{equation}

Although perhaps not immediately familiar, this leads to interesting digital representations. For base $-2$, for example, we have coefficients not of powers of 2, but of powers of $-2$: 1, $-2$, 4, $-8$, etc. 
While we can represent every non-negative integer with a standard $d$-ary representation, we cannot represent any of the negative integers. Using non-standard representations, we can. It is the case that representations which can represent every integer are given by one-component ECSDs:

\begin{thm}\label{integerrepresentation}
Let $G_S$ be an ECSD with representative set $S = \{(a_i, D) \in \Z^2 : 1 \leq i \leq d\}$, the base $D = \pm d$, and $d \geq 2 \in \Z$. Then every integer $n$ can be expressed as the sum 
\begin{equation} \label{representationsum}
    \sum_{j=0}^k b_jD^j : b_j \in \{a_1, \ldots, a_d\}
\end{equation}
for some $k \in \N$ if and only if $G_S$ is connected and its single cycle contains 0. This representation is unique up to leading instances of the block of digits representing the path from 0 to itself in $G_S$. 
\end{thm}

\begin{proof}
As in the previous discussion, the set of descendants of 0 in $G_S$ is the set of $n \in \Z$ that can be expressed as (\ref{representationsum}) for some $k \in \N$. Thus, if there is only one cycle (and therefore only one component by Proposition \ref{componentcycles}) and 0 is a cyclic vertex, every integer must be a descendant of 0 and can be expressed as (\ref{representationsum}) for some $k \in \N$. If $G_S$ has more than one component or if $0$ is not a cyclic vertex, then some integers are not descendants of 0 and hence do not have such a representation.

Because the structure of each component of an ECSD is a single base cycle and infinite trees at each cyclic vertex, a walk from $0$ to $n \in \Z$ will become unique once the path leaves the cycle. Since representations correspond to walks from $0$ to $n$, the only difference may be an arbitrary number of walks around the cycle which correspond to leading digits. 
\end{proof}

We are used to digital representations being unique up to leading 0 digits; for example, in the standard binary representation, we would say ``00011" = ``11". If 0 is the only cyclic vertex of an ECSD, then the cycle is a loop at 0, so the path from 0 to itself is one edge with corresponding digit 0. Thus, in this case we ignore leading zeroes as usual. 

\begin{figure}[h]
    \centering
    \begin{tikzpicture}[x=0.75pt,y=0.75pt,yscale=-1,xscale=1]

\draw (284,137) node    {$1$};

\draw (337,137) node    {$2$};

\draw (309,177) node    {$0$};

\draw (309,222) node    {$4$};

\draw (248,106) node    {$-1$};

\draw (373,108) node    {$-3$};

\draw (278,259) node    {$-7$};

\draw (337,257) node    {$-4$};

\draw (200,105) node    {$3$};

\draw (247,62) node    {$6$};

\draw (385,64) node    {$7$};

\draw (417,105) node    {$10$};

\draw (226,293) node  [rotate=-39.17]  {$\vdots $};

\draw (280,306) node  [rotate=-180.22]  {$\vdots $};

\draw (333,303) node  [rotate=-4.36]  {$\vdots $};

\draw (383,287) node  [rotate=-125.5]  {$\vdots $};

\draw (460,122) node  [rotate=-113.78]  {$\vdots $};

\draw (459,80) node  [rotate=-64.5]  {$\vdots $};

\draw (432,52) node  [rotate=-72.1]  {$\vdots $};

\draw (385,23) node  [rotate=-4.36]  {$\vdots $};

\draw (280,23) node  [rotate=-30.51]  {$\vdots $};

\draw (212,26) node  [rotate=-320.33]  {$\vdots $};

\draw (174,65) node  [rotate=-322.2]  {$\vdots $};

\draw (157,114) node  [rotate=-74.37]  {$\vdots $};

\draw [color={rgb, 255:red, 255; green, 0; blue, 31 }  ,draw opacity=1 ]   (293.5,137) -- (325.5,137) ;
\draw [shift={(327.5,137)}, rotate = 180] [color={rgb, 255:red, 255; green, 0; blue, 31 }  ,draw opacity=1 ][line width=0.75]    (10.93,-3.29) .. controls (6.95,-1.4) and (3.31,-0.3) .. (0,0) .. controls (3.31,0.3) and (6.95,1.4) .. (10.93,3.29)   ;

\draw [color={rgb, 255:red, 255; green, 0; blue, 31 }  ,draw opacity=1 ]   (328.6,149) -- (318.55,163.36) ;
\draw [shift={(317.4,165)}, rotate = 304.99] [color={rgb, 255:red, 255; green, 0; blue, 31 }  ,draw opacity=1 ][line width=0.75]    (10.93,-3.29) .. controls (6.95,-1.4) and (3.31,-0.3) .. (0,0) .. controls (3.31,0.3) and (6.95,1.4) .. (10.93,3.29)   ;

\draw [color={rgb, 255:red, 0; green, 115; blue, 255 }  ,draw opacity=1 ]   (301.5,165) -- (292.56,150.7) ;
\draw [shift={(291.5,149)}, rotate = 417.99] [color={rgb, 255:red, 0; green, 115; blue, 255 }  ,draw opacity=1 ][line width=0.75]    (10.93,-3.29) .. controls (6.95,-1.4) and (3.31,-0.3) .. (0,0) .. controls (3.31,0.3) and (6.95,1.4) .. (10.93,3.29)   ;

\draw [color={rgb, 255:red, 0; green, 115; blue, 255 }  ,draw opacity=1 ]   (274.5,128.82) -- (263.02,118.93) ;
\draw [shift={(261.5,117.63)}, rotate = 400.73] [color={rgb, 255:red, 0; green, 115; blue, 255 }  ,draw opacity=1 ][line width=0.75]    (10.93,-3.29) .. controls (6.95,-1.4) and (3.31,-0.3) .. (0,0) .. controls (3.31,0.3) and (6.95,1.4) .. (10.93,3.29)   ;

\draw [color={rgb, 255:red, 0; green, 115; blue, 255 }  ,draw opacity=1 ]   (346.5,129.35) -- (357.94,120.13) ;
\draw [shift={(359.5,118.88)}, rotate = 501.15] [color={rgb, 255:red, 0; green, 115; blue, 255 }  ,draw opacity=1 ][line width=0.75]    (10.93,-3.29) .. controls (6.95,-1.4) and (3.31,-0.3) .. (0,0) .. controls (3.31,0.3) and (6.95,1.4) .. (10.93,3.29)   ;

\draw [color={rgb, 255:red, 255; green, 0; blue, 31 }  ,draw opacity=1 ]   (309,189) -- (309,208) ;
\draw [shift={(309,210)}, rotate = 270] [color={rgb, 255:red, 255; green, 0; blue, 31 }  ,draw opacity=1 ][line width=0.75]    (10.93,-3.29) .. controls (6.95,-1.4) and (3.31,-0.3) .. (0,0) .. controls (3.31,0.3) and (6.95,1.4) .. (10.93,3.29)   ;

\draw [color={rgb, 255:red, 255; green, 0; blue, 31 }  ,draw opacity=1 ]   (318.5,233.88) -- (326.15,243.44) ;
\draw [shift={(327.4,245)}, rotate = 231.34] [color={rgb, 255:red, 255; green, 0; blue, 31 }  ,draw opacity=1 ][line width=0.75]    (10.93,-3.29) .. controls (6.95,-1.4) and (3.31,-0.3) .. (0,0) .. controls (3.31,0.3) and (6.95,1.4) .. (10.93,3.29)   ;

\draw [color={rgb, 255:red, 0; green, 115; blue, 255 }  ,draw opacity=1 ]   (299.5,233.34) -- (289.34,245.47) ;
\draw [shift={(288.05,247)}, rotate = 309.96000000000004] [color={rgb, 255:red, 0; green, 115; blue, 255 }  ,draw opacity=1 ][line width=0.75]    (10.93,-3.29) .. controls (6.95,-1.4) and (3.31,-0.3) .. (0,0) .. controls (3.31,0.3) and (6.95,1.4) .. (10.93,3.29)   ;

\draw [color={rgb, 255:red, 255; green, 0; blue, 31 }  ,draw opacity=1 ]   (247.73,94) -- (247.32,76) ;
\draw [shift={(247.27,74)}, rotate = 448.7] [color={rgb, 255:red, 255; green, 0; blue, 31 }  ,draw opacity=1 ][line width=0.75]    (10.93,-3.29) .. controls (6.95,-1.4) and (3.31,-0.3) .. (0,0) .. controls (3.31,0.3) and (6.95,1.4) .. (10.93,3.29)   ;

\draw [color={rgb, 255:red, 0; green, 115; blue, 255 }  ,draw opacity=1 ]   (234.5,105.72) -- (211.5,105.24) ;
\draw [shift={(209.5,105.2)}, rotate = 361.19] [color={rgb, 255:red, 0; green, 115; blue, 255 }  ,draw opacity=1 ][line width=0.75]    (10.93,-3.29) .. controls (6.95,-1.4) and (3.31,-0.3) .. (0,0) .. controls (3.31,0.3) and (6.95,1.4) .. (10.93,3.29)   ;

\draw [color={rgb, 255:red, 255; green, 0; blue, 31 }  ,draw opacity=1 ]   (386.5,107.08) -- (402.5,105.99) ;
\draw [shift={(404.5,105.85)}, rotate = 536.1] [color={rgb, 255:red, 255; green, 0; blue, 31 }  ,draw opacity=1 ][line width=0.75]    (10.93,-3.29) .. controls (6.95,-1.4) and (3.31,-0.3) .. (0,0) .. controls (3.31,0.3) and (6.95,1.4) .. (10.93,3.29)   ;

\draw [color={rgb, 255:red, 0; green, 115; blue, 255 }  ,draw opacity=1 ]   (376.27,96) -- (381.2,77.93) ;
\draw [shift={(381.73,76)}, rotate = 465.26] [color={rgb, 255:red, 0; green, 115; blue, 255 }  ,draw opacity=1 ][line width=0.75]    (10.93,-3.29) .. controls (6.95,-1.4) and (3.31,-0.3) .. (0,0) .. controls (3.31,0.3) and (6.95,1.4) .. (10.93,3.29)   ;

\draw [color={rgb, 255:red, 255; green, 0; blue, 31 }  ,draw opacity=1 ]   (264.5,267.83) -- (243.36,281.65) ;
\draw [shift={(241.68,282.75)}, rotate = 326.82] [color={rgb, 255:red, 255; green, 0; blue, 31 }  ,draw opacity=1 ][line width=0.75]    (10.93,-3.29) .. controls (6.95,-1.4) and (3.31,-0.3) .. (0,0) .. controls (3.31,0.3) and (6.95,1.4) .. (10.93,3.29)   ;

\draw [color={rgb, 255:red, 0; green, 115; blue, 255 }  ,draw opacity=1 ]   (278.51,271) -- (279.4,291.97) ;
\draw [shift={(279.49,293.96)}, rotate = 267.56] [color={rgb, 255:red, 0; green, 115; blue, 255 }  ,draw opacity=1 ][line width=0.75]    (10.93,-3.29) .. controls (6.95,-1.4) and (3.31,-0.3) .. (0,0) .. controls (3.31,0.3) and (6.95,1.4) .. (10.93,3.29)   ;

\draw [color={rgb, 255:red, 255; green, 0; blue, 31 }  ,draw opacity=1 ]   (335.96,269) -- (334.28,288.31) ;
\draw [shift={(334.1,290.31)}, rotate = 274.97] [color={rgb, 255:red, 255; green, 0; blue, 31 }  ,draw opacity=1 ][line width=0.75]    (10.93,-3.29) .. controls (6.95,-1.4) and (3.31,-0.3) .. (0,0) .. controls (3.31,0.3) and (6.95,1.4) .. (10.93,3.29)   ;

\draw [color={rgb, 255:red, 0; green, 115; blue, 255 }  ,draw opacity=1 ]   (350.5,265.8) -- (365.67,275.7) ;
\draw [shift={(367.34,276.79)}, rotate = 213.11] [color={rgb, 255:red, 0; green, 115; blue, 255 }  ,draw opacity=1 ][line width=0.75]    (10.93,-3.29) .. controls (6.95,-1.4) and (3.31,-0.3) .. (0,0) .. controls (3.31,0.3) and (6.95,1.4) .. (10.93,3.29)   ;

\draw [color={rgb, 255:red, 255; green, 0; blue, 31 }  ,draw opacity=1 ]   (429.5,109.94) -- (443.16,115.34) ;
\draw [shift={(445.02,116.08)}, rotate = 201.57] [color={rgb, 255:red, 255; green, 0; blue, 31 }  ,draw opacity=1 ][line width=0.75]    (10.93,-3.29) .. controls (6.95,-1.4) and (3.31,-0.3) .. (0,0) .. controls (3.31,0.3) and (6.95,1.4) .. (10.93,3.29)   ;

\draw [color={rgb, 255:red, 0; green, 115; blue, 255 }  ,draw opacity=1 ]   (429.5,97.56) -- (442.17,90.02) ;
\draw [shift={(443.88,89)}, rotate = 509.24] [color={rgb, 255:red, 0; green, 115; blue, 255 }  ,draw opacity=1 ][line width=0.75]    (10.93,-3.29) .. controls (6.95,-1.4) and (3.31,-0.3) .. (0,0) .. controls (3.31,0.3) and (6.95,1.4) .. (10.93,3.29)   ;

\draw [color={rgb, 255:red, 255; green, 0; blue, 31 }  ,draw opacity=1 ]   (394.5,61.57) -- (415.63,56.18) ;
\draw [shift={(417.56,55.69)}, rotate = 525.6800000000001] [color={rgb, 255:red, 255; green, 0; blue, 31 }  ,draw opacity=1 ][line width=0.75]    (10.93,-3.29) .. controls (6.95,-1.4) and (3.31,-0.3) .. (0,0) .. controls (3.31,0.3) and (6.95,1.4) .. (10.93,3.29)   ;

\draw [color={rgb, 255:red, 0; green, 115; blue, 255 }  ,draw opacity=1 ]   (385,52) -- (385,37.69) ;
\draw [shift={(385,35.69)}, rotate = 450] [color={rgb, 255:red, 0; green, 115; blue, 255 }  ,draw opacity=1 ][line width=0.75]    (10.93,-3.29) .. controls (6.95,-1.4) and (3.31,-0.3) .. (0,0) .. controls (3.31,0.3) and (6.95,1.4) .. (10.93,3.29)   ;

\draw [color={rgb, 255:red, 255; green, 0; blue, 31 }  ,draw opacity=1 ]   (256.5,50.77) -- (265.64,39.97) ;
\draw [shift={(266.94,38.44)}, rotate = 490.24] [color={rgb, 255:red, 255; green, 0; blue, 31 }  ,draw opacity=1 ][line width=0.75]    (10.93,-3.29) .. controls (6.95,-1.4) and (3.31,-0.3) .. (0,0) .. controls (3.31,0.3) and (6.95,1.4) .. (10.93,3.29)   ;

\draw [color={rgb, 255:red, 0; green, 115; blue, 255 }  ,draw opacity=1 ]   (237.5,52.23) -- (228.72,43.2) ;
\draw [shift={(227.32,41.76)}, rotate = 405.81] [color={rgb, 255:red, 0; green, 115; blue, 255 }  ,draw opacity=1 ][line width=0.75]    (10.93,-3.29) .. controls (6.95,-1.4) and (3.31,-0.3) .. (0,0) .. controls (3.31,0.3) and (6.95,1.4) .. (10.93,3.29)   ;

\draw [color={rgb, 255:red, 255; green, 0; blue, 31 }  ,draw opacity=1 ]   (192.2,93) -- (185.31,82.4) ;
\draw [shift={(184.22,80.72)}, rotate = 416.98] [color={rgb, 255:red, 255; green, 0; blue, 31 }  ,draw opacity=1 ][line width=0.75]    (10.93,-3.29) .. controls (6.95,-1.4) and (3.31,-0.3) .. (0,0) .. controls (3.31,0.3) and (6.95,1.4) .. (10.93,3.29)   ;

\draw [color={rgb, 255:red, 0; green, 115; blue, 255 }  ,draw opacity=1 ]   (190.5,106.99) -- (173.15,110.62) ;
\draw [shift={(171.19,111.03)}, rotate = 348.18] [color={rgb, 255:red, 0; green, 115; blue, 255 }  ,draw opacity=1 ][line width=0.75]    (10.93,-3.29) .. controls (6.95,-1.4) and (3.31,-0.3) .. (0,0) .. controls (3.31,0.3) and (6.95,1.4) .. (10.93,3.29)   ;

\end{tikzpicture}
    \caption{The ECSD $G(-2n+1, -2n+4)$ with $(n, -2n+1)$ edges colored blue and $(n, -2n+4)$ edges colored red.}
    \label{G(-2n+1,-2n+4)rb}
\end{figure}
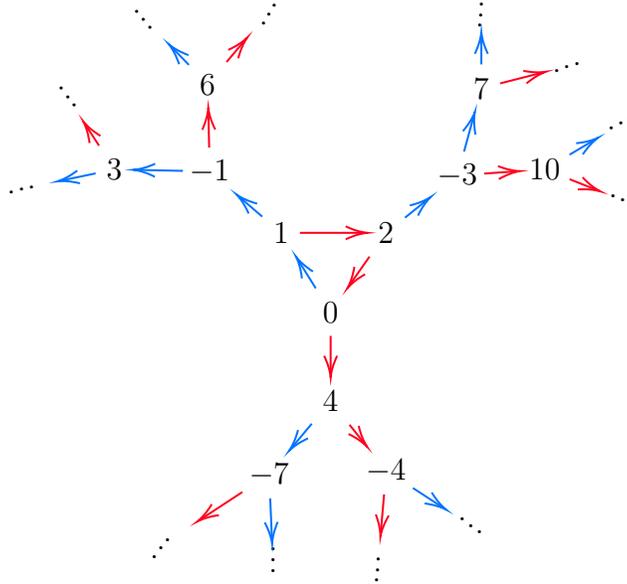

However, consider a more complicated case, such as $G(-2n+1, -2n+4)$ (shown in Figure \ref{G(-2n+1,-2n+4)rb} with $(n, -2n+1)$ edges colored blue and $(n, -2n+4)$ edges colored red). One can reach -3 by a blue step, a red step, and a blue step from 0, so the shortest representation of $-3$ would be ``$141$", as $1(-2)^2+4(-2)^1 + 1(-2)^0 = -3$. However, we can also reach $-3$ by going once around the cycle first, giving us a different representation of ``$144141$", as 
\[1(-2)^5 + 4(-2)^4 + 4(-2)^3 + 1(-2)^2+4(-2)^1 + 1(-2)^0 = -3.\] Each time we go around the cycle at the beginning, we add an extra ``144" to the beginning of the representation. This is because the shortest representation of $0$ in this digital system is ``$144$", since 0 is not a digit: we have that $1(-2)^2 + 4(-2)^1 + 4(-2)^0 = 0$. 
Thus, by ignoring leading appearances of ``$144$", we are really ignoring leading zeroes. \\

\noindent
The idea of non-standard representations is certainly not new, though this specific way of studying them seems to be. 

Arithmetic in negative bases was first proposed by Gr\"{u}nwald in 1885 \cite{Grunwald}, though his work seemed to be forgotten for a time.  With the advent of digital computing in the mid 20th century, interest in negative bases exploded, especially in regards to base $-2$ which became known as ``negabinary" (term first introduced by deRegt \cite{deRegt}). Numerous papers have been published on the subject, but are all about the computer science and efficiency of such number systems \cite{PawlakWakulicz}. Such a system was more efficient, because instead of needing a bit to represent whether a number was negative or positive, every \textit{integer} could be represented in such a manner, and its sign depended on the number of digits. It is important to note that all of these publications assumed a standard digit set. 

The most commonly explored type of alternate digit set is a \emph{symmetric} digit set, in which for an odd base $d$, the digits are taken to be $ -\frac{r-1}{2}, \ldots, \frac{r-1}{2} $. Such systems seem to be as old as civilization itself, for example, one can see similarities with the number system of the Yoruba people of West Africa \cite{Zaslavsky}. In such a system, one counts \textit{down} from a larger place value as well as up. 

As well as using an odd base and taking digits centered about 0, in fact, one can use any integer base $d \geq 3$ and any set of consecutive $d-1$ digits including $-1$, $0$, and $1$, and get a system in which one can represent all integers, not just all natural numbers. The first modern study of such systems is Colson in 1726 \cite{Colson}, who studied ``negativo-affirmative arithmetick," a base 10 system using the digits $\{-4, -3, \ldots, 4, 5 \}$. In his book \textit{The Philosophy of Arithmetic}, Leslie includes such an idea, suggesting the same digits as Colson, though acknowledging there could be other choices of digit set. \cite{Leslie}. Cauchy suggested the extension to an arbitrary base in 1840 \cite{Cauchy}, and Lalanne provided a follow-up discussion in the same year \cite{Lalanne}. A century later, in 1950, Shannon published an informative summary in the \textit{American Mathematical Monthly} \cite{Shannon}. In most of these publications, emphasis was put on the ease of calculating with such a number system --- it effectively gets rid of larger digits, and the same computations can be used for multiple calculations. 

One system in particular that holds great interest for computer scientists is \textit{balanced ternary}, which Knuth termed ``perhaps the prettiest number system of all:" a number system in base 3 with digit set $\{-1,0,1\}$ \cite{Knuth}. Of course, this is just a special case of a symmetric digit set. In the following proposition we use ECSDs to prove that every integer can be expressed in balanced ternary: 

\begin{prop}
Every integer can be written uniquely as $\displaystyle \sum_{i=0}^k \epsilon_i3^i$ for $\epsilon_i \in \{0, \pm 1\}$, where 
$ \epsilon_k \neq 0$.
\end{prop}

\begin{proof} It suffices to show that 0 is the only cyclic vertex of $G(3n, 3n-1, 3n+1)$, from which Theorem \ref{integerrepresentation} completes the proof. The ECSD $G(3n, 3n-1, 3n+1)$ is shown in Figure \ref{G(3n,3n+1,3n-1)}. 

Each cyclic vertex $m$ must have itself as a $k$-descendant for some $k \in \N$. However, if $m > 0$, its three successors $3m$, $3m+1$, and $3m-1$ are all greater than $m$, so none of its descendants will be $m$. Similarly, if $m < 0$, all of its descendants are less than $m$. Thus, 0 is the only cyclic vertex of $G(3n, 3n-1, 3n+1)$. By Theorem \ref{integerrepresentation}, each representation is unique up to leading zeroes.
\end{proof}

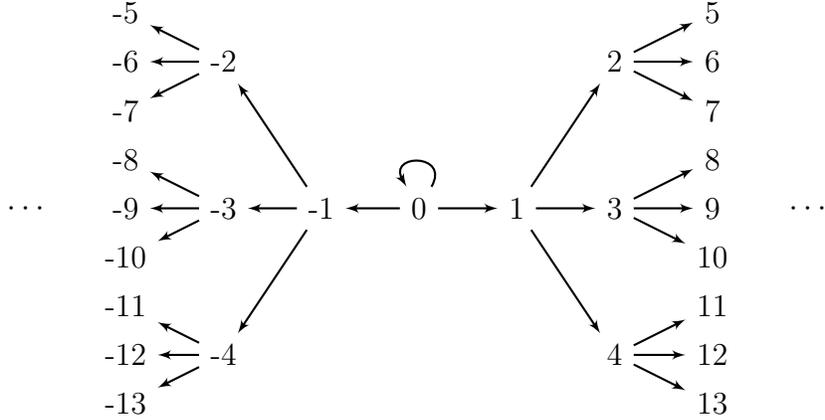
\begin{figure}[h]
    \centering
    \begin{tikzpicture}[scale=1.3]
    \node (0) at (0,0) {$0$};
    \draw[edge,draw=black,thick] (0) to[out=60,in=120,looseness=4] (0);
    \node (1) at (1,0) {1};
    \node (-1) at (-1,0) {-1};
    \edge{0}{1}
    \edge{0}{-1}
    
    \node (2) at (2,1.5) {2};
    \node (3) at (2,0) {3};
    \node (4) at (2,-1.5) {4};
    \edge{1}{2}
    \edge{1}{3}
    \edge{1}{4}
    
    \node (-2) at (-2,1.5) {-2};
    \node (-3) at (-2,0) {-3};
    \node (-4) at (-2,-1.5) {-4};
    \edge{-1}{-2}
    \edge{-1}{-3}
    \edge{-1}{-4}
    
    \node (5) at (3, 2) {5};
    \node (6) at (3, 1.5) {6};
    \node (7) at (3, 1) {7};
    \edge{2}{5}
    \edge{2}{6}
    \edge{2}{7}

    \node (8) at (3, .5) {8};
    \node (9) at (3, 0) {9};
    \node (10) at (3, -.5) {10};
    \edge{3}{8}
    \edge{3}{9}
    \edge{3}{10}
    
    \node (11) at (3, -1) {11};
    \node (12) at (3, -1.5) {12};
    \node (13) at (3, -2) {13};
    \edge{4}{11}
    \edge{4}{12}
    \edge{4}{13}
    
    \node (-5) at (-3, 2) {-5};
    \node (-6) at (-3, 1.5) {-6};
    \node (-7) at (-3, 1) {-7};
    \edge{-2}{-5}
    \edge{-2}{-6}
    \edge{-2}{-7}

    \node (-8) at (-3, .5) {-8};
    \node (-9) at (-3, 0) {-9};
    \node (-10) at (-3, -.5) {-10};
    \edge{-3}{-8}
    \edge{-3}{-9}
    \edge{-3}{-10}
    
    \node (-11) at (-3, -1) {-11};
    \node (-12) at (-3, -1.5) {-12};
    \node (-13) at (-3, -2) {-13};
    \edge{-4}{-11}
    \edge{-4}{-12}
    \edge{-4}{-13}
    
    \node at (4, 0) {$\cdots$};
    \node at (-4, 0) {$\cdots$};
\end{tikzpicture}
    \caption{The ECSD $G(3n, 3n+1, 3n-1)$.}
    \label{G(3n,3n+1,3n-1)}
\end{figure}

We can extend this proof to the general case for a consecutive digit set including digits $-1$, 0, and 1 for base $d \geq 3$:

\begin{prop}[\cite{Colson}] \label{nonstandarddaryrep}
Let $d \geq 3$. If $-(d-1) < t < 0$, then every integer can be written uniquely as \[\displaystyle \sum_{i=0}^k b_i d^i \text{ for } b_i \in \{t, t+1, \ldots, 0 \ldots, t+d-1\} \text{ where } b_k \neq 0.\] 
\end{prop}

\begin{proof}
It suffices to show that 0 is the only cyclic vertex of \[G(dn+ t, dn+(t+1), \ldots, dn, \ldots, dn+ t+d-1),\] from which Theorem \ref{integerrepresentation} completes the proof. We can see that 0 is a loop in this ECSD because $dm = m$ if and only if $m = 0$. If $m \geq 1$, all of its successors are greater than $m$, because the smallest successor is $dm+t > dm -(d-1) \geq m$, so it cannot be a descendant of itself. If $m \leq -1$, all of its successors are smaller than $m$, because the largest successor is $dm+t+d-1 < dm+d-1 \leq m$, so it cannot be a descendant of itself. Therefore 0 is the only cyclic vertex.
\end{proof}

\begin{remark}
In the previous proposition, if $t=0$, we get the standard $d$-ary representation, that is, every non-negative integer can be expressed uniquely as \[\sum_{i=0}^k b_i d^i\text{ for }b_i \in \{0, 1, 2, \ldots, d-1\}.\] If $t = -(d-1)$, we get the standard representation but multiplied by $-1$, that is, we can represent every non-positive integer as $\sum_{i=0}^k b_i d^i$ for $b_i \in \{0, -1, -2, \ldots, -(d-1)\}$. Both of these can be proved in a similar manner to the proof of Theorem \ref{nonstandarddaryrep} using ECSDs: each corresponding ECSD has two components, and the component with 0 as the sole cyclic vertex contains all the non-negative integers in the standard case and all the non-positive integers in the second case. 
\end{remark}

Notably, this digital representation does not work for $d=2$, because there is no valid integer choice of $t$ between $-(d-1) = -1$ and $0$. However, we can represent every integer in base ``2" using the \textit{nega}binary representation, as follows: 

\begin{prop}[\cite{Grunwald}]
Every integer can be written uniquely as $\displaystyle \sum_{i=0}^\infty b_i(-2)^i$ for $b_i \in \{0,1\}$. 
\end{prop}

The ECSD $G(-2n, -2n+1)$ has a single component and 0 is its only cyclic vertex. It is shown in Figure \ref{G(-2n,-2n+1)}, where we can see that the rows alternate between all negative and all positive values: the direct $k$-descendants of 0 are positive for odd $k$ and negative for even $k$. 

\begin{figure}[h]
\centering
\begin{tikzpicture}[x=0.75pt,y=0.75pt,yscale=-1,xscale=1]

\draw [color={rgb, 255:red, 0; green, 0; blue, 0 }  ,draw opacity=1 ]   (281,38.5) .. controls (260.42,10.08) and (311.87,10.47) .. (298.88,35.91) ;
\draw [shift={(298,37.5)}, rotate = 300.65] [color={rgb, 255:red, 0; green, 0; blue, 0 }  ,draw opacity=1 ][line width=0.75]    (10.93,-3.29) .. controls (6.95,-1.4) and (3.31,-0.3) .. (0,0) .. controls (3.31,0.3) and (6.95,1.4) .. (10.93,3.29)   ;

\draw (290.17,44.83) node    {$0$};

\draw (290.17,91.33) node    {$1$};

\draw (224,132.33) node    {$-1$};

\draw (355.33,131) node    {$-2$};

\draw (197,170.33) node    {$2$};

\draw (256.67,170.33) node    {$3$};

\draw (324,170.33) node    {$4$};

\draw (382.5,170.33) node    {$5$};

\draw (181.5,210.67) node    {$-3$};

\draw (210.5,210.67) node    {$-4$};

\draw (241.33,210.67) node    {$-5$};

\draw (270.17,210.67) node    {$-6$};

\draw (308.5,210.67) node    {$-7$};

\draw (339.83,210.67) node    {$-8$};

\draw (368.5,210.67) node    {$-9$};

\draw (401.33,210.67) node    {$-10$};

\draw (171,252) node    {$\vdots $};

\draw (188.41,252) node    {$\vdots $};

\draw (204.23,252) node    {$\vdots $};

\draw (220.05,252) node    {$\vdots $};

\draw (235.87,252) node    {$\vdots $};

\draw (251.7,252) node    {$\vdots $};

\draw (267.52,252) node    {$\vdots $};

\draw (283.34,252) node    {$\vdots $};

\draw (299.16,252) node    {$\vdots $};

\draw (314.98,252) node    {$\vdots $};

\draw (330.8,252) node    {$\vdots $};

\draw (346.62,252) node    {$\vdots $};

\draw (362.45,252) node    {$\vdots $};

\draw (378.27,252) node    {$\vdots $};

\draw (394.09,252) node    {$\vdots $};

\draw (411.5,252) node    {$\vdots $};

\draw    (290.17,56.83) -- (290.17,77.33) ;
\draw [shift={(290.17,79.33)}, rotate = 270] [color={rgb, 255:red, 0; green, 0; blue, 0 }  ][line width=0.75]    (10.93,-3.29) .. controls (6.95,-1.4) and (3.31,-0.3) .. (0,0) .. controls (3.31,0.3) and (6.95,1.4) .. (10.93,3.29)   ;

\draw    (299.67,97.12) -- (340.12,121.74) ;
\draw [shift={(341.83,122.78)}, rotate = 211.32999999999998] [color={rgb, 255:red, 0; green, 0; blue, 0 }  ][line width=0.75]    (10.93,-3.29) .. controls (6.95,-1.4) and (3.31,-0.3) .. (0,0) .. controls (3.31,0.3) and (6.95,1.4) .. (10.93,3.29)   ;

\draw    (280.67,97.22) -- (239.2,122.91) ;
\draw [shift={(237.5,123.97)}, rotate = 328.22] [color={rgb, 255:red, 0; green, 0; blue, 0 }  ][line width=0.75]    (10.93,-3.29) .. controls (6.95,-1.4) and (3.31,-0.3) .. (0,0) .. controls (3.31,0.3) and (6.95,1.4) .. (10.93,3.29)   ;

\draw    (234.32,144.33) -- (245.86,157.77) ;
\draw [shift={(247.17,159.28)}, rotate = 229.32] [color={rgb, 255:red, 0; green, 0; blue, 0 }  ][line width=0.75]    (10.93,-3.29) .. controls (6.95,-1.4) and (3.31,-0.3) .. (0,0) .. controls (3.31,0.3) and (6.95,1.4) .. (10.93,3.29)   ;

\draw    (215.47,144.33) -- (206.68,156.7) ;
\draw [shift={(205.53,158.33)}, rotate = 305.39] [color={rgb, 255:red, 0; green, 0; blue, 0 }  ][line width=0.75]    (10.93,-3.29) .. controls (6.95,-1.4) and (3.31,-0.3) .. (0,0) .. controls (3.31,0.3) and (6.95,1.4) .. (10.93,3.29)   ;

\draw    (345.77,143) -- (334.75,156.84) ;
\draw [shift={(333.5,158.41)}, rotate = 308.53999999999996] [color={rgb, 255:red, 0; green, 0; blue, 0 }  ][line width=0.75]    (10.93,-3.29) .. controls (6.95,-1.4) and (3.31,-0.3) .. (0,0) .. controls (3.31,0.3) and (6.95,1.4) .. (10.93,3.29)   ;

\draw    (363.62,143) -- (373.08,156.69) ;
\draw [shift={(374.21,158.33)}, rotate = 235.37] [color={rgb, 255:red, 0; green, 0; blue, 0 }  ][line width=0.75]    (10.93,-3.29) .. controls (6.95,-1.4) and (3.31,-0.3) .. (0,0) .. controls (3.31,0.3) and (6.95,1.4) .. (10.93,3.29)   ;

\draw    (192.39,182.33) -- (186.83,196.8) ;
\draw [shift={(186.11,198.67)}, rotate = 291.02] [color={rgb, 255:red, 0; green, 0; blue, 0 }  ][line width=0.75]    (10.93,-3.29) .. controls (6.95,-1.4) and (3.31,-0.3) .. (0,0) .. controls (3.31,0.3) and (6.95,1.4) .. (10.93,3.29)   ;

\draw    (201.02,182.33) -- (205.85,196.77) ;
\draw [shift={(206.48,198.67)}, rotate = 251.49] [color={rgb, 255:red, 0; green, 0; blue, 0 }  ][line width=0.75]    (10.93,-3.29) .. controls (6.95,-1.4) and (3.31,-0.3) .. (0,0) .. controls (3.31,0.3) and (6.95,1.4) .. (10.93,3.29)   ;

\draw    (252.1,182.33) -- (246.61,196.8) ;
\draw [shift={(245.9,198.67)}, rotate = 290.82] [color={rgb, 255:red, 0; green, 0; blue, 0 }  ][line width=0.75]    (10.93,-3.29) .. controls (6.95,-1.4) and (3.31,-0.3) .. (0,0) .. controls (3.31,0.3) and (6.95,1.4) .. (10.93,3.29)   ;

\draw    (260.68,182.33) -- (265.52,196.77) ;
\draw [shift={(266.15,198.67)}, rotate = 251.49] [color={rgb, 255:red, 0; green, 0; blue, 0 }  ][line width=0.75]    (10.93,-3.29) .. controls (6.95,-1.4) and (3.31,-0.3) .. (0,0) .. controls (3.31,0.3) and (6.95,1.4) .. (10.93,3.29)   ;

\draw    (319.39,182.33) -- (313.83,196.8) ;
\draw [shift={(313.11,198.67)}, rotate = 291.02] [color={rgb, 255:red, 0; green, 0; blue, 0 }  ][line width=0.75]    (10.93,-3.29) .. controls (6.95,-1.4) and (3.31,-0.3) .. (0,0) .. controls (3.31,0.3) and (6.95,1.4) .. (10.93,3.29)   ;

\draw    (328.71,182.33) -- (334.39,196.8) ;
\draw [shift={(335.12,198.67)}, rotate = 248.57] [color={rgb, 255:red, 0; green, 0; blue, 0 }  ][line width=0.75]    (10.93,-3.29) .. controls (6.95,-1.4) and (3.31,-0.3) .. (0,0) .. controls (3.31,0.3) and (6.95,1.4) .. (10.93,3.29)   ;

\draw    (378.33,182.33) -- (373.32,196.78) ;
\draw [shift={(372.67,198.67)}, rotate = 289.14] [color={rgb, 255:red, 0; green, 0; blue, 0 }  ][line width=0.75]    (10.93,-3.29) .. controls (6.95,-1.4) and (3.31,-0.3) .. (0,0) .. controls (3.31,0.3) and (6.95,1.4) .. (10.93,3.29)   ;

\draw    (388.1,182.33) -- (394.88,196.85) ;
\draw [shift={(395.73,198.67)}, rotate = 244.97] [color={rgb, 255:red, 0; green, 0; blue, 0 }  ][line width=0.75]    (10.93,-3.29) .. controls (6.95,-1.4) and (3.31,-0.3) .. (0,0) .. controls (3.31,0.3) and (6.95,1.4) .. (10.93,3.29)   ;

\draw    (178.45,222.67) -- (174.54,238.06) ;
\draw [shift={(174.05,240)}, rotate = 284.25] [color={rgb, 255:red, 0; green, 0; blue, 0 }  ][line width=0.75]    (10.93,-3.29) .. controls (6.95,-1.4) and (3.31,-0.3) .. (0,0) .. controls (3.31,0.3) and (6.95,1.4) .. (10.93,3.29)   ;

\draw    (183.51,222.67) -- (186.07,238.03) ;
\draw [shift={(186.4,240)}, rotate = 260.51] [color={rgb, 255:red, 0; green, 0; blue, 0 }  ][line width=0.75]    (10.93,-3.29) .. controls (6.95,-1.4) and (3.31,-0.3) .. (0,0) .. controls (3.31,0.3) and (6.95,1.4) .. (10.93,3.29)   ;

\draw    (208.68,222.67) -- (206.35,238.02) ;
\draw [shift={(206.05,240)}, rotate = 278.62] [color={rgb, 255:red, 0; green, 0; blue, 0 }  ][line width=0.75]    (10.93,-3.29) .. controls (6.95,-1.4) and (3.31,-0.3) .. (0,0) .. controls (3.31,0.3) and (6.95,1.4) .. (10.93,3.29)   ;

\draw    (213.27,222.67) -- (216.83,238.05) ;
\draw [shift={(217.28,240)}, rotate = 256.99] [color={rgb, 255:red, 0; green, 0; blue, 0 }  ][line width=0.75]    (10.93,-3.29) .. controls (6.95,-1.4) and (3.31,-0.3) .. (0,0) .. controls (3.31,0.3) and (6.95,1.4) .. (10.93,3.29)   ;

\draw    (239.75,222.67) -- (237.72,238.02) ;
\draw [shift={(237.46,240)}, rotate = 277.52] [color={rgb, 255:red, 0; green, 0; blue, 0 }  ][line width=0.75]    (10.93,-3.29) .. controls (6.95,-1.4) and (3.31,-0.3) .. (0,0) .. controls (3.31,0.3) and (6.95,1.4) .. (10.93,3.29)   ;

\draw    (244.34,222.67) -- (248.2,238.06) ;
\draw [shift={(248.69,240)}, rotate = 255.92000000000002] [color={rgb, 255:red, 0; green, 0; blue, 0 }  ][line width=0.75]    (10.93,-3.29) .. controls (6.95,-1.4) and (3.31,-0.3) .. (0,0) .. controls (3.31,0.3) and (6.95,1.4) .. (10.93,3.29)   ;

\draw    (269.4,222.67) -- (268.41,238) ;
\draw [shift={(268.29,240)}, rotate = 273.67] [color={rgb, 255:red, 0; green, 0; blue, 0 }  ][line width=0.75]    (10.93,-3.29) .. controls (6.95,-1.4) and (3.31,-0.3) .. (0,0) .. controls (3.31,0.3) and (6.95,1.4) .. (10.93,3.29)   ;

\draw    (273.99,222.67) -- (278.91,238.09) ;
\draw [shift={(279.51,240)}, rotate = 252.32] [color={rgb, 255:red, 0; green, 0; blue, 0 }  ][line width=0.75]    (10.93,-3.29) .. controls (6.95,-1.4) and (3.31,-0.3) .. (0,0) .. controls (3.31,0.3) and (6.95,1.4) .. (10.93,3.29)   ;

\draw    (310.38,222.67) -- (312.79,238.02) ;
\draw [shift={(313.1,240)}, rotate = 261.09000000000003] [color={rgb, 255:red, 0; green, 0; blue, 0 }  ][line width=0.75]    (10.93,-3.29) .. controls (6.95,-1.4) and (3.31,-0.3) .. (0,0) .. controls (3.31,0.3) and (6.95,1.4) .. (10.93,3.29)   ;

\draw    (305.79,222.67) -- (302.31,238.05) ;
\draw [shift={(301.87,240)}, rotate = 282.73] [color={rgb, 255:red, 0; green, 0; blue, 0 }  ][line width=0.75]    (10.93,-3.29) .. controls (6.95,-1.4) and (3.31,-0.3) .. (0,0) .. controls (3.31,0.3) and (6.95,1.4) .. (10.93,3.29)   ;

\draw    (337.21,222.67) -- (333.85,238.05) ;
\draw [shift={(333.43,240)}, rotate = 282.32] [color={rgb, 255:red, 0; green, 0; blue, 0 }  ][line width=0.75]    (10.93,-3.29) .. controls (6.95,-1.4) and (3.31,-0.3) .. (0,0) .. controls (3.31,0.3) and (6.95,1.4) .. (10.93,3.29)   ;

\draw    (341.81,222.67) -- (344.33,238.03) ;
\draw [shift={(344.65,240)}, rotate = 260.67] [color={rgb, 255:red, 0; green, 0; blue, 0 }  ][line width=0.75]    (10.93,-3.29) .. controls (6.95,-1.4) and (3.31,-0.3) .. (0,0) .. controls (3.31,0.3) and (6.95,1.4) .. (10.93,3.29)   ;

\draw    (366.74,222.67) -- (364.49,238.02) ;
\draw [shift={(364.2,240)}, rotate = 278.33] [color={rgb, 255:red, 0; green, 0; blue, 0 }  ][line width=0.75]    (10.93,-3.29) .. controls (6.95,-1.4) and (3.31,-0.3) .. (0,0) .. controls (3.31,0.3) and (6.95,1.4) .. (10.93,3.29)   ;

\draw    (371.34,222.67) -- (374.97,238.05) ;
\draw [shift={(375.43,240)}, rotate = 256.7] [color={rgb, 255:red, 0; green, 0; blue, 0 }  ][line width=0.75]    (10.93,-3.29) .. controls (6.95,-1.4) and (3.31,-0.3) .. (0,0) .. controls (3.31,0.3) and (6.95,1.4) .. (10.93,3.29)   ;

\draw    (399.23,222.67) -- (396.54,238.03) ;
\draw [shift={(396.19,240)}, rotate = 279.94] [color={rgb, 255:red, 0; green, 0; blue, 0 }  ][line width=0.75]    (10.93,-3.29) .. controls (6.95,-1.4) and (3.31,-0.3) .. (0,0) .. controls (3.31,0.3) and (6.95,1.4) .. (10.93,3.29)   ;

\draw    (404.28,222.67) -- (408.07,238.06) ;
\draw [shift={(408.55,240)}, rotate = 256.18] [color={rgb, 255:red, 0; green, 0; blue, 0 }  ][line width=0.75]    (10.93,-3.29) .. controls (6.95,-1.4) and (3.31,-0.3) .. (0,0) .. controls (3.31,0.3) and (6.95,1.4) .. (10.93,3.29)   ;

\end{tikzpicture}
\caption{The ECSD $G(-2n, -2n+1)$.}
\label{G(-2n,-2n+1)}
\end{figure}
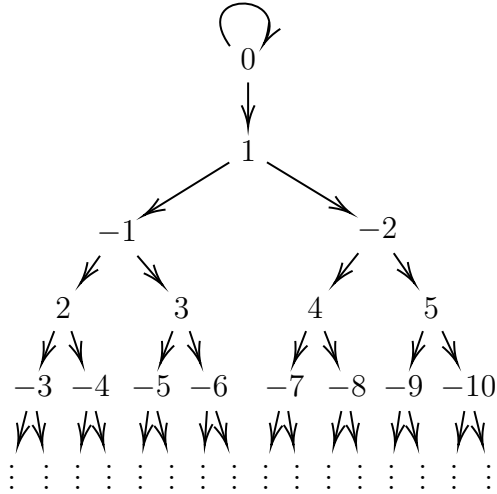

We can see that the structure of this ECSD will hold for a generalization into ECSDs of higher degree, giving us the same result for negative-base digital representations of degree $d$. For example, the ECSD $G(-3n, -3n+1, -3n+2)$ will have the same general structure, containing every integer and alternating between positive and negative rows, shown in Figure \ref{G(-3n,-3n+1,-3n+2)}. Thus, we can prove the following: 

\begin{figure}[h]
\centering
\begin{tikzpicture}[x=0.75pt,y=0.75pt,yscale=-.75,xscale=.75]

\draw [color={rgb, 255:red, 0; green, 0; blue, 0 }  ,draw opacity=1 ]   (308,31.5) .. controls (287.42,3.08) and (338.87,3.47) .. (325.88,28.91) ;
\draw [shift={(325,30.5)}, rotate = 300.65] [color={rgb, 255:red, 0; green, 0; blue, 0 }  ,draw opacity=1 ][line width=0.75]    (10.93,-3.29) .. controls (6.95,-1.4) and (3.31,-0.3) .. (0,0) .. controls (3.31,0.3) and (6.95,1.4) .. (10.93,3.29)   ;

\draw (317.17,37.83) node    {$0$};

\draw (39.36,194) node    {$\vdots $};

\draw (66.6,193) node    {$\vdots $};

\draw (95.36,193) node    {$\vdots $};

\draw (129.11,193) node    {$\vdots $};

\draw (157.86,193) node    {$\vdots $};

\draw (191.37,193) node    {$\vdots $};

\draw (221.12,193) node    {$\vdots $};

\draw (262.04,194) node    {$\vdots $};

\draw (294.64,194) node    {$\vdots $};

\draw (333.25,194) node    {$\vdots $};

\draw (368.85,194) node    {$\vdots $};

\draw (404.45,194) node    {$\vdots $};

\draw (443.05,194) node    {$\vdots $};

\draw (479.66,194) node    {$\vdots $};

\draw (517.84,194) node    {$\vdots $};

\draw (164,84.33) node    {$1$};

\draw (480,84.33) node    {$2$};

\draw (68,127) node    {$-1$};

\draw (162.5,127) node    {$-2$};

\draw (258,130) node    {$-3$};

\draw (368,127) node    {$-4$};

\draw (480,127) node    {$-5$};

\draw (592,130) node    {$-6$};

\draw (39,170.33) node    {$3$};

\draw (67.13,170.33) node    {$4$};

\draw (98.38,170.33) node    {$5$};

\draw (129.63,170.33) node    {$6$};

\draw (160.88,170.33) node    {$7$};

\draw (192.13,170.33) node    {$8$};

\draw (223.38,170.33) node    {$9$};

\draw (257.63,170.33) node    {$10$};

\draw (294.88,170.33) node    {$11$};

\draw (332.13,170.33) node    {$12$};

\draw (369.38,170.33) node    {$13$};

\draw (406.63,170.33) node    {$14$};

\draw (443.88,170.33) node    {$15$};

\draw (481.13,170.33) node    {$16$};

\draw (518.38,170.33) node    {$17$};

\draw (555.63,170.33) node    {$18$};

\draw (592.88,170.33) node    {$19$};

\draw (627,170.33) node    {$20$};

\draw (554.84,194) node    {$\vdots $};

\draw (593.84,194) node    {$\vdots $};

\draw (624.84,194) node    {$\vdots $};

\draw    (307.67,40.72) -- (175.41,80.87) ;
\draw [shift={(173.5,81.45)}, rotate = 343.11] [color={rgb, 255:red, 0; green, 0; blue, 0 }  ][line width=0.75]    (10.93,-3.29) .. controls (6.95,-1.4) and (3.31,-0.3) .. (0,0) .. controls (3.31,0.3) and (6.95,1.4) .. (10.93,3.29)   ;

\draw    (154.5,88.56) -- (83.33,120.19) ;
\draw [shift={(81.5,121)}, rotate = 336.03999999999996] [color={rgb, 255:red, 0; green, 0; blue, 0 }  ][line width=0.75]    (10.93,-3.29) .. controls (6.95,-1.4) and (3.31,-0.3) .. (0,0) .. controls (3.31,0.3) and (6.95,1.4) .. (10.93,3.29)   ;

\draw    (163.58,96.33) -- (162.99,113) ;
\draw [shift={(162.92,115)}, rotate = 272.01] [color={rgb, 255:red, 0; green, 0; blue, 0 }  ][line width=0.75]    (10.93,-3.29) .. controls (6.95,-1.4) and (3.31,-0.3) .. (0,0) .. controls (3.31,0.3) and (6.95,1.4) .. (10.93,3.29)   ;

\draw    (173.5,88.95) -- (242.7,122.57) ;
\draw [shift={(244.5,123.44)}, rotate = 205.91] [color={rgb, 255:red, 0; green, 0; blue, 0 }  ][line width=0.75]    (10.93,-3.29) .. controls (6.95,-1.4) and (3.31,-0.3) .. (0,0) .. controls (3.31,0.3) and (6.95,1.4) .. (10.93,3.29)   ;

\draw    (326.67,40.55) -- (468.58,81.07) ;
\draw [shift={(470.5,81.62)}, rotate = 195.94] [color={rgb, 255:red, 0; green, 0; blue, 0 }  ][line width=0.75]    (10.93,-3.29) .. controls (6.95,-1.4) and (3.31,-0.3) .. (0,0) .. controls (3.31,0.3) and (6.95,1.4) .. (10.93,3.29)   ;

\draw    (470.5,87.95) -- (383.37,121.15) ;
\draw [shift={(381.5,121.86)}, rotate = 339.15] [color={rgb, 255:red, 0; green, 0; blue, 0 }  ][line width=0.75]    (10.93,-3.29) .. controls (6.95,-1.4) and (3.31,-0.3) .. (0,0) .. controls (3.31,0.3) and (6.95,1.4) .. (10.93,3.29)   ;

\draw    (480,96.33) -- (480,113) ;
\draw [shift={(480,115)}, rotate = 270] [color={rgb, 255:red, 0; green, 0; blue, 0 }  ][line width=0.75]    (10.93,-3.29) .. controls (6.95,-1.4) and (3.31,-0.3) .. (0,0) .. controls (3.31,0.3) and (6.95,1.4) .. (10.93,3.29)   ;

\draw    (489.5,88.21) -- (576.65,123.74) ;
\draw [shift={(578.5,124.5)}, rotate = 202.18] [color={rgb, 255:red, 0; green, 0; blue, 0 }  ][line width=0.75]    (10.93,-3.29) .. controls (6.95,-1.4) and (3.31,-0.3) .. (0,0) .. controls (3.31,0.3) and (6.95,1.4) .. (10.93,3.29)   ;

\draw    (59.97,139) -- (48.14,156.67) ;
\draw [shift={(47.03,158.33)}, rotate = 303.79] [color={rgb, 255:red, 0; green, 0; blue, 0 }  ][line width=0.75]    (10.93,-3.29) .. controls (6.95,-1.4) and (3.31,-0.3) .. (0,0) .. controls (3.31,0.3) and (6.95,1.4) .. (10.93,3.29)   ;

\draw    (67.76,139) -- (67.41,156.33) ;
\draw [shift={(67.37,158.33)}, rotate = 271.15999999999997] [color={rgb, 255:red, 0; green, 0; blue, 0 }  ][line width=0.75]    (10.93,-3.29) .. controls (6.95,-1.4) and (3.31,-0.3) .. (0,0) .. controls (3.31,0.3) and (6.95,1.4) .. (10.93,3.29)   ;

\draw    (76.41,139) -- (88.82,156.7) ;
\draw [shift={(89.96,158.33)}, rotate = 234.97] [color={rgb, 255:red, 0; green, 0; blue, 0 }  ][line width=0.75]    (10.93,-3.29) .. controls (6.95,-1.4) and (3.31,-0.3) .. (0,0) .. controls (3.31,0.3) and (6.95,1.4) .. (10.93,3.29)   ;

\draw    (153.4,139) -- (139.94,156.74) ;
\draw [shift={(138.73,158.33)}, rotate = 307.19] [color={rgb, 255:red, 0; green, 0; blue, 0 }  ][line width=0.75]    (10.93,-3.29) .. controls (6.95,-1.4) and (3.31,-0.3) .. (0,0) .. controls (3.31,0.3) and (6.95,1.4) .. (10.93,3.29)   ;

\draw    (162.05,139) -- (161.4,156.33) ;
\draw [shift={(161.32,158.33)}, rotate = 272.15] [color={rgb, 255:red, 0; green, 0; blue, 0 }  ][line width=0.75]    (10.93,-3.29) .. controls (6.95,-1.4) and (3.31,-0.3) .. (0,0) .. controls (3.31,0.3) and (6.95,1.4) .. (10.93,3.29)   ;

\draw    (170.7,139) -- (182.79,156.68) ;
\draw [shift={(183.92,158.33)}, rotate = 235.64] [color={rgb, 255:red, 0; green, 0; blue, 0 }  ][line width=0.75]    (10.93,-3.29) .. controls (6.95,-1.4) and (3.31,-0.3) .. (0,0) .. controls (3.31,0.3) and (6.95,1.4) .. (10.93,3.29)   ;

\draw    (247.7,142) -- (234.18,157.75) ;
\draw [shift={(232.88,159.27)}, rotate = 310.65] [color={rgb, 255:red, 0; green, 0; blue, 0 }  ][line width=0.75]    (10.93,-3.29) .. controls (6.95,-1.4) and (3.31,-0.3) .. (0,0) .. controls (3.31,0.3) and (6.95,1.4) .. (10.93,3.29)   ;

\draw    (257.89,142) -- (257.76,156.33) ;
\draw [shift={(257.74,158.33)}, rotate = 270.53] [color={rgb, 255:red, 0; green, 0; blue, 0 }  ][line width=0.75]    (10.93,-3.29) .. controls (6.95,-1.4) and (3.31,-0.3) .. (0,0) .. controls (3.31,0.3) and (6.95,1.4) .. (10.93,3.29)   ;

\draw    (268.97,142) -- (282.55,156.86) ;
\draw [shift={(283.9,158.33)}, rotate = 227.56] [color={rgb, 255:red, 0; green, 0; blue, 0 }  ][line width=0.75]    (10.93,-3.29) .. controls (6.95,-1.4) and (3.31,-0.3) .. (0,0) .. controls (3.31,0.3) and (6.95,1.4) .. (10.93,3.29)   ;

\draw    (358.07,139) -- (343.34,156.79) ;
\draw [shift={(342.06,158.33)}, rotate = 309.62] [color={rgb, 255:red, 0; green, 0; blue, 0 }  ][line width=0.75]    (10.93,-3.29) .. controls (6.95,-1.4) and (3.31,-0.3) .. (0,0) .. controls (3.31,0.3) and (6.95,1.4) .. (10.93,3.29)   ;

\draw    (368.38,139) -- (368.93,156.33) ;
\draw [shift={(368.99,158.33)}, rotate = 268.18] [color={rgb, 255:red, 0; green, 0; blue, 0 }  ][line width=0.75]    (10.93,-3.29) .. controls (6.95,-1.4) and (3.31,-0.3) .. (0,0) .. controls (3.31,0.3) and (6.95,1.4) .. (10.93,3.29)   ;

\draw    (378.7,139) -- (394.6,156.84) ;
\draw [shift={(395.93,158.33)}, rotate = 228.29] [color={rgb, 255:red, 0; green, 0; blue, 0 }  ][line width=0.75]    (10.93,-3.29) .. controls (6.95,-1.4) and (3.31,-0.3) .. (0,0) .. controls (3.31,0.3) and (6.95,1.4) .. (10.93,3.29)   ;

\draw    (470,139) -- (455.16,156.8) ;
\draw [shift={(453.88,158.33)}, rotate = 309.82] [color={rgb, 255:red, 0; green, 0; blue, 0 }  ][line width=0.75]    (10.93,-3.29) .. controls (6.95,-1.4) and (3.31,-0.3) .. (0,0) .. controls (3.31,0.3) and (6.95,1.4) .. (10.93,3.29)   ;

\draw    (480.31,139) -- (480.76,156.33) ;
\draw [shift={(480.81,158.33)}, rotate = 268.51] [color={rgb, 255:red, 0; green, 0; blue, 0 }  ][line width=0.75]    (10.93,-3.29) .. controls (6.95,-1.4) and (3.31,-0.3) .. (0,0) .. controls (3.31,0.3) and (6.95,1.4) .. (10.93,3.29)   ;

\draw    (490.63,139) -- (506.42,156.84) ;
\draw [shift={(507.75,158.33)}, rotate = 228.47] [color={rgb, 255:red, 0; green, 0; blue, 0 }  ][line width=0.75]    (10.93,-3.29) .. controls (6.95,-1.4) and (3.31,-0.3) .. (0,0) .. controls (3.31,0.3) and (6.95,1.4) .. (10.93,3.29)   ;

\draw    (581.18,142) -- (567.79,156.85) ;
\draw [shift={(566.45,158.33)}, rotate = 312.05] [color={rgb, 255:red, 0; green, 0; blue, 0 }  ][line width=0.75]    (10.93,-3.29) .. controls (6.95,-1.4) and (3.31,-0.3) .. (0,0) .. controls (3.31,0.3) and (6.95,1.4) .. (10.93,3.29)   ;

\draw    (592.26,142) -- (592.57,156.33) ;
\draw [shift={(592.61,158.33)}, rotate = 268.76] [color={rgb, 255:red, 0; green, 0; blue, 0 }  ][line width=0.75]    (10.93,-3.29) .. controls (6.95,-1.4) and (3.31,-0.3) .. (0,0) .. controls (3.31,0.3) and (6.95,1.4) .. (10.93,3.29)   ;

\draw    (602.41,142) -- (615.28,156.82) ;
\draw [shift={(616.59,158.33)}, rotate = 229.05] [color={rgb, 255:red, 0; green, 0; blue, 0 }  ][line width=0.75]    (10.93,-3.29) .. controls (6.95,-1.4) and (3.31,-0.3) .. (0,0) .. controls (3.31,0.3) and (6.95,1.4) .. (10.93,3.29)   ;

\end{tikzpicture}\caption{The ECSD $G(-3n,-3n+1,-3n+2)$.}
\label{G(-3n,-3n+1,-3n+2)}
\end{figure}
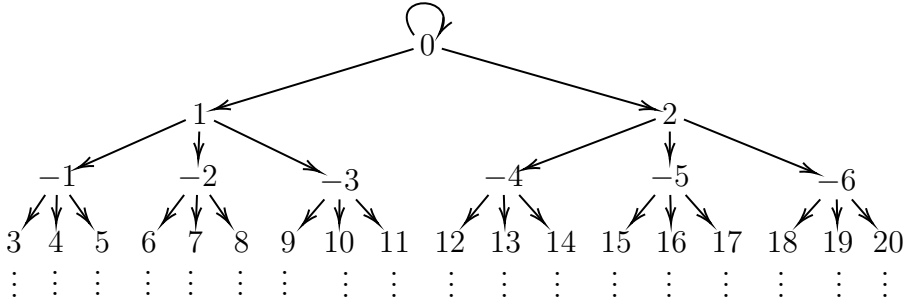

\begin{prop}[\cite{Grunwald}]
For any natural number $d \geq 2$, every integer can be written uniquely as $\displaystyle \sum_{i=0}^\infty b_i(-d)^i$ for $b_i \in \{0, 1, \ldots, d-1\}$.
\end{prop}

\begin{proof}
It again suffices to show that 0 is the only cyclic vertex of \[G(-dn, -dn + 1, \ldots, -dn + d -1).\] 

Suppose $m > 0$. Then  its successors are negative and its $2$-descendants are positive. Its largest successor is $-dm+d-1$, so its smallest $2$-descendant is \[-d(-dm+d-1) = d^2(m-1)+d > m\] since $d \geq 2$ and $m > 0$. It follows that all of the positive descendants of $m$ are greater than $m$, thus $m$ is not a cyclic vertex.

Suppose $m < 0$. Then its successors are positive and its 2-descendants are negative. Its smallest successor is $-dm$, so its largest 2-descendant is $-d(-dm)+d-1 = d^2m + d-1 < m$ since $d \geq 2$ and $m < 0$. It follows that all of the negative descendants of $m$ are less than $m$, and thus that $m$ is not a cyclic vertex.

Since $-d(0)=0$, we have that 0 is a cyclic vertex, so $m=0$ is the only cyclic vertex. 
\end{proof}

In later years the topic of numeration systems was extended to consider any base, positive or negative, along with any digit set --- however, the digit set is always defined to include the digit 0 \cite{AlloucheShallit}. This is because of the leading-zeroes problem, which we have already accounted for in Theorem \ref{integerrepresentation}: by ignoring leading appearances of the block of digits corresponding to the base cycle, this is equivalent to ignoring leading zeroes. 

If we expand our definitions to allow digit sets that do not include 0, we can find other non-standard digits sets for a negabinary representation that can represent all integers. This task is equivalent to finding ECSDs of degree 2 with one component. We can show that $G(-2n+1, -2n+4)$, depicted in Figure \ref{G(-2n+1,-2n+4)rb}, has a single component. We know that every cycle of an ECSD must contain a vertex $m$ such that $|P(m)| \geq |m|$, and by Lemma \ref{ancestorlemma}, vertices with this property are bounded by some $N \in \N$. Thus, it is simply a task of extending the graph until we have all of these vertices to determine the number of components. In the case of $G(-2n+1, -2n+4)$, $N=4$ by the proof of Lemma \ref{ancestorlemma}. Since Figure \ref{G(-2n+1,-2n+4)rb} includes all integers in the range $-4$ to $4$  (except for $-2$ which is a successor of 3), we can see that this is indeed the only component of this ECSD. By this same process, we can also show that the ECSD $G(-2n+1, -2n+10)$ has a single component, as depicted in Figure \ref{G(-2n+1,-2n+10)}.

\begin{figure}[h]
\centering
\begin{tikzpicture}[x=0.75pt,y=0.75pt,yscale=-1,xscale=1]

\draw (218,205) node    {$1$};

\draw (246,252) node    {$-1$};

\draw (302,270) node    {$3$};

\draw (380,207) node    {$2$};

\draw (348,252) node    {$4$};

\draw (370,153) node    {$6$};

\draw (331,115) node    {$-2$};

\draw (269,116) node    {$5$};

\draw (227,152) node    {$0$};

\draw (180,130) node    {$10$};

\draw (168,219) node    {$8$};

\draw (219,295) node    {$12$};

\draw (302,315) node    {$-5$};

\draw (384,293) node    {$-7$};

\draw (429,218) node    {$-3$};

\draw (419,137) node    {$-11$};

\draw (353,68) node    {$14$};

\draw (252,71) node    {$-9$};

\draw (156,88) node  [rotate=-331.57]  {$\vdots $};

\draw (135,147) node  [rotate=-250.07]  {$\vdots $};

\draw (128,191) node  [rotate=-294.15]  {$\vdots $};

\draw (135,258) node  [rotate=-223.15]  {$\vdots $};

\draw (175,308) node  [rotate=-244.28]  {$\vdots $};

\draw (219,343) node    {$\vdots $};

\draw (280,351) node  [rotate=-25.19]  {$\vdots $};

\draw (336,355) node  [rotate=-142.88]  {$\vdots $};

\draw (388,338) node  [rotate=-180]  {$\vdots $};

\draw (430,296) node  [rotate=-95.35]  {$\vdots $};

\draw (472,255) node  [rotate=-129.72]  {$\vdots $};

\draw (477,191) node  [rotate=-54.5]  {$\vdots $};

\draw (468,153) node  [rotate=-107.89]  {$\vdots $};

\draw (454,94) node  [rotate=-39.27]  {$\vdots $};

\draw (397,60) node  [rotate=-82.19]  {$\vdots $};

\draw (349,27) node  [rotate=-354.79]  {$\vdots $};

\draw (281,27) node  [rotate=-32.55]  {$\vdots $};

\draw (207,32) node  [rotate=-309.88]  {$\vdots $};

\draw    (225.15,217) -- (237.83,238.28) ;
\draw [shift={(238.85,240)}, rotate = 239.22] [color={rgb, 255:red, 0; green, 0; blue, 0 }  ][line width=0.75]    (10.93,-3.29) .. controls (6.95,-1.4) and (3.31,-0.3) .. (0,0) .. controls (3.31,0.3) and (6.95,1.4) .. (10.93,3.29)   ;

\draw    (259.5,256.34) -- (290.6,266.33) ;
\draw [shift={(292.5,266.95)}, rotate = 197.82] [color={rgb, 255:red, 0; green, 0; blue, 0 }  ][line width=0.75]    (10.93,-3.29) .. controls (6.95,-1.4) and (3.31,-0.3) .. (0,0) .. controls (3.31,0.3) and (6.95,1.4) .. (10.93,3.29)   ;

\draw    (224.96,164) -- (220.37,191.03) ;
\draw [shift={(220.04,193)}, rotate = 279.64] [color={rgb, 255:red, 0; green, 0; blue, 0 }  ][line width=0.75]    (10.93,-3.29) .. controls (6.95,-1.4) and (3.31,-0.3) .. (0,0) .. controls (3.31,0.3) and (6.95,1.4) .. (10.93,3.29)   ;

\draw    (259.5,124.14) -- (238.02,142.56) ;
\draw [shift={(236.5,143.86)}, rotate = 319.4] [color={rgb, 255:red, 0; green, 0; blue, 0 }  ][line width=0.75]    (10.93,-3.29) .. controls (6.95,-1.4) and (3.31,-0.3) .. (0,0) .. controls (3.31,0.3) and (6.95,1.4) .. (10.93,3.29)   ;

\draw    (317.5,115.22) -- (280.5,115.81) ;
\draw [shift={(278.5,115.85)}, rotate = 359.08000000000004] [color={rgb, 255:red, 0; green, 0; blue, 0 }  ][line width=0.75]    (10.93,-3.29) .. controls (6.95,-1.4) and (3.31,-0.3) .. (0,0) .. controls (3.31,0.3) and (6.95,1.4) .. (10.93,3.29)   ;

\draw    (360.5,143.74) -- (344.75,128.4) ;
\draw [shift={(343.32,127)}, rotate = 404.26] [color={rgb, 255:red, 0; green, 0; blue, 0 }  ][line width=0.75]    (10.93,-3.29) .. controls (6.95,-1.4) and (3.31,-0.3) .. (0,0) .. controls (3.31,0.3) and (6.95,1.4) .. (10.93,3.29)   ;

\draw    (377.78,195) -- (372.59,166.97) ;
\draw [shift={(372.22,165)}, rotate = 439.51] [color={rgb, 255:red, 0; green, 0; blue, 0 }  ][line width=0.75]    (10.93,-3.29) .. controls (6.95,-1.4) and (3.31,-0.3) .. (0,0) .. controls (3.31,0.3) and (6.95,1.4) .. (10.93,3.29)   ;

\draw    (356.53,240) -- (370.31,220.63) ;
\draw [shift={(371.47,219)}, rotate = 485.42] [color={rgb, 255:red, 0; green, 0; blue, 0 }  ][line width=0.75]    (10.93,-3.29) .. controls (6.95,-1.4) and (3.31,-0.3) .. (0,0) .. controls (3.31,0.3) and (6.95,1.4) .. (10.93,3.29)   ;

\draw    (311.5,266.28) -- (336.64,256.45) ;
\draw [shift={(338.5,255.72)}, rotate = 518.63] [color={rgb, 255:red, 0; green, 0; blue, 0 }  ][line width=0.75]    (10.93,-3.29) .. controls (6.95,-1.4) and (3.31,-0.3) .. (0,0) .. controls (3.31,0.3) and (6.95,1.4) .. (10.93,3.29)   ;

\draw    (217.5,147.55) -- (194.31,136.7) ;
\draw [shift={(192.5,135.85)}, rotate = 385.08000000000004] [color={rgb, 255:red, 0; green, 0; blue, 0 }  ][line width=0.75]    (10.93,-3.29) .. controls (6.95,-1.4) and (3.31,-0.3) .. (0,0) .. controls (3.31,0.3) and (6.95,1.4) .. (10.93,3.29)   ;

\draw    (208.5,207.66) -- (179.43,215.8) ;
\draw [shift={(177.5,216.34)}, rotate = 344.36] [color={rgb, 255:red, 0; green, 0; blue, 0 }  ][line width=0.75]    (10.93,-3.29) .. controls (6.95,-1.4) and (3.31,-0.3) .. (0,0) .. controls (3.31,0.3) and (6.95,1.4) .. (10.93,3.29)   ;

\draw    (238.47,264) -- (227.6,281.31) ;
\draw [shift={(226.53,283)}, rotate = 302.12] [color={rgb, 255:red, 0; green, 0; blue, 0 }  ][line width=0.75]    (10.93,-3.29) .. controls (6.95,-1.4) and (3.31,-0.3) .. (0,0) .. controls (3.31,0.3) and (6.95,1.4) .. (10.93,3.29)   ;

\draw    (302,282) -- (302,301) ;
\draw [shift={(302,303)}, rotate = 270] [color={rgb, 255:red, 0; green, 0; blue, 0 }  ][line width=0.75]    (10.93,-3.29) .. controls (6.95,-1.4) and (3.31,-0.3) .. (0,0) .. controls (3.31,0.3) and (6.95,1.4) .. (10.93,3.29)   ;

\draw    (357.5,262.82) -- (372.14,279.5) ;
\draw [shift={(373.46,281)}, rotate = 228.72] [color={rgb, 255:red, 0; green, 0; blue, 0 }  ][line width=0.75]    (10.93,-3.29) .. controls (6.95,-1.4) and (3.31,-0.3) .. (0,0) .. controls (3.31,0.3) and (6.95,1.4) .. (10.93,3.29)   ;

\draw    (389.5,209.13) -- (413.55,214.53) ;
\draw [shift={(415.5,214.97)}, rotate = 192.65] [color={rgb, 255:red, 0; green, 0; blue, 0 }  ][line width=0.75]    (10.93,-3.29) .. controls (6.95,-1.4) and (3.31,-0.3) .. (0,0) .. controls (3.31,0.3) and (6.95,1.4) .. (10.93,3.29)   ;

\draw    (379.5,149.9) -- (399.1,143.5) ;
\draw [shift={(401,142.88)}, rotate = 521.9200000000001] [color={rgb, 255:red, 0; green, 0; blue, 0 }  ][line width=0.75]    (10.93,-3.29) .. controls (6.95,-1.4) and (3.31,-0.3) .. (0,0) .. controls (3.31,0.3) and (6.95,1.4) .. (10.93,3.29)   ;

\draw    (336.62,103) -- (346.54,81.81) ;
\draw [shift={(347.38,80)}, rotate = 475.08] [color={rgb, 255:red, 0; green, 0; blue, 0 }  ][line width=0.75]    (10.93,-3.29) .. controls (6.95,-1.4) and (3.31,-0.3) .. (0,0) .. controls (3.31,0.3) and (6.95,1.4) .. (10.93,3.29)   ;

\draw    (264.47,104) -- (257.24,84.87) ;
\draw [shift={(256.53,83)}, rotate = 429.3] [color={rgb, 255:red, 0; green, 0; blue, 0 }  ][line width=0.75]    (10.93,-3.29) .. controls (6.95,-1.4) and (3.31,-0.3) .. (0,0) .. controls (3.31,0.3) and (6.95,1.4) .. (10.93,3.29)   ;

\draw    (173.14,118) -- (165.74,105.05) ;
\draw [shift={(164.75,103.32)}, rotate = 420.26] [color={rgb, 255:red, 0; green, 0; blue, 0 }  ][line width=0.75]    (10.93,-3.29) .. controls (6.95,-1.4) and (3.31,-0.3) .. (0,0) .. controls (3.31,0.3) and (6.95,1.4) .. (10.93,3.29)   ;

\draw    (167.5,134.72) -- (151.51,140.76) ;
\draw [shift={(149.64,141.47)}, rotate = 339.3] [color={rgb, 255:red, 0; green, 0; blue, 0 }  ][line width=0.75]    (10.93,-3.29) .. controls (6.95,-1.4) and (3.31,-0.3) .. (0,0) .. controls (3.31,0.3) and (6.95,1.4) .. (10.93,3.29)   ;

\draw    (158.5,212.35) -- (144.65,202.65) ;
\draw [shift={(143.01,201.51)}, rotate = 394.99] [color={rgb, 255:red, 0; green, 0; blue, 0 }  ][line width=0.75]    (10.93,-3.29) .. controls (6.95,-1.4) and (3.31,-0.3) .. (0,0) .. controls (3.31,0.3) and (6.95,1.4) .. (10.93,3.29)   ;

\draw    (158.5,230.23) -- (149.66,240.68) ;
\draw [shift={(148.36,242.21)}, rotate = 310.24] [color={rgb, 255:red, 0; green, 0; blue, 0 }  ][line width=0.75]    (10.93,-3.29) .. controls (6.95,-1.4) and (3.31,-0.3) .. (0,0) .. controls (3.31,0.3) and (6.95,1.4) .. (10.93,3.29)   ;

\draw    (206.5,298.69) -- (192.05,302.96) ;
\draw [shift={(190.13,303.53)}, rotate = 343.53999999999996] [color={rgb, 255:red, 0; green, 0; blue, 0 }  ][line width=0.75]    (10.93,-3.29) .. controls (6.95,-1.4) and (3.31,-0.3) .. (0,0) .. controls (3.31,0.3) and (6.95,1.4) .. (10.93,3.29)   ;

\draw    (219,307) -- (219,329) ;
\draw [shift={(219,331)}, rotate = 270] [color={rgb, 255:red, 0; green, 0; blue, 0 }  ][line width=0.75]    (10.93,-3.29) .. controls (6.95,-1.4) and (3.31,-0.3) .. (0,0) .. controls (3.31,0.3) and (6.95,1.4) .. (10.93,3.29)   ;

\draw    (294.67,327) -- (290.27,334.2) ;
\draw [shift={(289.22,335.91)}, rotate = 301.43] [color={rgb, 255:red, 0; green, 0; blue, 0 }  ][line width=0.75]    (10.93,-3.29) .. controls (6.95,-1.4) and (3.31,-0.3) .. (0,0) .. controls (3.31,0.3) and (6.95,1.4) .. (10.93,3.29)   ;

\draw    (312.2,327) -- (321.35,337.77) ;
\draw [shift={(322.65,339.29)}, rotate = 229.64] [color={rgb, 255:red, 0; green, 0; blue, 0 }  ][line width=0.75]    (10.93,-3.29) .. controls (6.95,-1.4) and (3.31,-0.3) .. (0,0) .. controls (3.31,0.3) and (6.95,1.4) .. (10.93,3.29)   ;

\draw    (385.07,305) -- (386.76,324.01) ;
\draw [shift={(386.93,326)}, rotate = 264.92] [color={rgb, 255:red, 0; green, 0; blue, 0 }  ][line width=0.75]    (10.93,-3.29) .. controls (6.95,-1.4) and (3.31,-0.3) .. (0,0) .. controls (3.31,0.3) and (6.95,1.4) .. (10.93,3.29)   ;

\draw    (397.5,293.88) -- (415.16,295.03) ;
\draw [shift={(417.16,295.16)}, rotate = 183.73] [color={rgb, 255:red, 0; green, 0; blue, 0 }  ][line width=0.75]    (10.93,-3.29) .. controls (6.95,-1.4) and (3.31,-0.3) .. (0,0) .. controls (3.31,0.3) and (6.95,1.4) .. (10.93,3.29)   ;

\draw    (442.5,229.62) -- (454.72,240.13) ;
\draw [shift={(456.24,241.44)}, rotate = 220.71] [color={rgb, 255:red, 0; green, 0; blue, 0 }  ][line width=0.75]    (10.93,-3.29) .. controls (6.95,-1.4) and (3.31,-0.3) .. (0,0) .. controls (3.31,0.3) and (6.95,1.4) .. (10.93,3.29)   ;

\draw    (442.5,210.41) -- (459.6,200.79) ;
\draw [shift={(461.34,199.81)}, rotate = 510.64] [color={rgb, 255:red, 0; green, 0; blue, 0 }  ][line width=0.75]    (10.93,-3.29) .. controls (6.95,-1.4) and (3.31,-0.3) .. (0,0) .. controls (3.31,0.3) and (6.95,1.4) .. (10.93,3.29)   ;

\draw    (437,142.88) -- (451.66,147.67) ;
\draw [shift={(453.57,148.29)}, rotate = 198.07999999999998] [color={rgb, 255:red, 0; green, 0; blue, 0 }  ][line width=0.75]    (10.93,-3.29) .. controls (6.95,-1.4) and (3.31,-0.3) .. (0,0) .. controls (3.31,0.3) and (6.95,1.4) .. (10.93,3.29)   ;

\draw    (428.77,125) -- (439.91,111.31) ;
\draw [shift={(441.18,109.75)}, rotate = 489.14] [color={rgb, 255:red, 0; green, 0; blue, 0 }  ][line width=0.75]    (10.93,-3.29) .. controls (6.95,-1.4) and (3.31,-0.3) .. (0,0) .. controls (3.31,0.3) and (6.95,1.4) .. (10.93,3.29)   ;

\draw    (365.5,65.73) -- (381.83,62.76) ;
\draw [shift={(383.8,62.4)}, rotate = 529.7] [color={rgb, 255:red, 0; green, 0; blue, 0 }  ][line width=0.75]    (10.93,-3.29) .. controls (6.95,-1.4) and (3.31,-0.3) .. (0,0) .. controls (3.31,0.3) and (6.95,1.4) .. (10.93,3.29)   ;

\draw    (351.83,56) -- (350.45,41.81) ;
\draw [shift={(350.25,39.82)}, rotate = 444.43] [color={rgb, 255:red, 0; green, 0; blue, 0 }  ][line width=0.75]    (10.93,-3.29) .. controls (6.95,-1.4) and (3.31,-0.3) .. (0,0) .. controls (3.31,0.3) and (6.95,1.4) .. (10.93,3.29)   ;

\draw    (259.91,59) -- (269.66,44.21) ;
\draw [shift={(270.76,42.54)}, rotate = 483.39] [color={rgb, 255:red, 0; green, 0; blue, 0 }  ][line width=0.75]    (10.93,-3.29) .. controls (6.95,-1.4) and (3.31,-0.3) .. (0,0) .. controls (3.31,0.3) and (6.95,1.4) .. (10.93,3.29)   ;

\draw    (238.5,59.3) -- (224.28,46.97) ;
\draw [shift={(222.77,45.66)}, rotate = 400.90999999999997] [color={rgb, 255:red, 0; green, 0; blue, 0 }  ][line width=0.75]    (10.93,-3.29) .. controls (6.95,-1.4) and (3.31,-0.3) .. (0,0) .. controls (3.31,0.3) and (6.95,1.4) .. (10.93,3.29)   ;

\end{tikzpicture}
\caption{The ECSD $G(-2n+1, -2n+10)$.}
\label{G(-2n+1,-2n+10)}
\end{figure}

Using the classification of all degree 2 ECSDs, we can prove \cite{Neidinger} that the only ECSDs of degree 2 to have a single component are either $G(2n, -2n+1)$ or are isomorphic to an ECSD of the form $G(-2n+1, -2n+a)$ with $a = \pm 3^m+1$ for some $m \in \N_0$. Thus, by Theorem \ref{integerrepresentation} we have that 
\[\Z = \left\{\sum_{j=0}^k b_j(-2)^j : b_j \in \{1, a\}, k \in \N \right\}\] 
if and only if $a = \pm3^m + 1$ for some $m \in \N_0$.

\noindent 
\textit{AMS 2010 Mathematics Subject Classification:} Primary 11A07,  Secondary 05C20, 05C63, 11A63, 11A67.\\

\noindent 
\textit{Keywords:} exact covering system, infinite directed graph, non-standard digital representation\\

\noindent
(Concerned with sequences \seqnum{A002487} and \seqnum{A110081}.)

\end{document}